\documentclass[11pt]{article}

\usepackage{amsmath,amsthm,amssymb}
\usepackage{graphicx}
\usepackage{enumitem}
\usepackage{hyperref}

\title{On ``hard stars'' in general relativity}
\author{Grigorios Fournodavlos and Volker Schlue}

\numberwithin{equation}{section}

\usepackage[raggedleft,small,bf]{titlesec}
\newcommand{\ud}{\mathrm{d}}

\newcommand{\rhot}{\tilde{\rho}}
\newcommand{\mt}{\tilde{m}}

\theoremstyle{plain}
\newtheorem{prop}{Proposition}[section]
\newtheorem{lemma}[prop]{Lemma}
\newtheorem{cor}[prop]{Corollary}
\newtheorem{conj}[prop]{Conjecture}

\newtheorem{rem}{Remark}[section]

\allowdisplaybreaks[1]
\begin{document}

\maketitle

\begin{abstract}
  We study spherically symmetric solutions to the Einstein-Euler equations which model an idealized relativistic neutron star surrounded by vacuum. These are barotropic fluids with a free boundary, governed by an equation of state which sets the speed of sound equal to the speed of light. We demonstrate the existence of a 1-parameter family of \emph{static} solutions, or ``hard stars,'' and describe their stability properties:
  First, we show that \emph{small} stars are a local minimum of the mass energy functional under variations which preserve the total number of particles. In particular, we prove that the second variation of the mass energy functional controls the ``mass aspect function.''
  Second, we derive the linearisation of the Euler-Einstein system around small stars in ``comoving coordinates,'' and prove a uniform boundedness statement for an energy, which is exactly at the level of the variational argument. Finally, we exhibit the existence of time periodic solutions to the linearised system, which shows that energy boundedness is optimal for this problem.
\end{abstract}

\tableofcontents

\section{Introduction}

A description of the two body problem in general relativity which goes beyond approximations is challenging in part because fairly little is known about the dynamics of \emph{extended} bodies.\footnote{Many of the existing approximation schemes in the physics literature are based on the approximation of bodies by point masses; see e.g.~\cite{Blanchet:14} and references therein.}
A physically relevant example of such a body is a neutron star.\footnote{Binary systems of neutron stars are -- besides black holes -- the primary objects of interest in gravitational wave astronomy; see e.g.~the recent \cite{ligo:pulsar}.}  While the question of the correct equation of state has been subject to much debate,\footnote{Specifically for the considerations that have motivated \textsc{Christodoulou}'s model, see e.g.~\cite{zeldovich,friedman}.} we focus here on an idealized description in the context of \textsc{Christodoulou}'s two-phase model \cite{Ch:95:I}, and are interested in the dynamical stability of even \emph{one} such body in \emph{spherical symmetry}.

In this model the body is described by an \emph{irrotational barotropic fluid} governed by an equation of state which sets the speed of sound equal to the speed of light. This so-called ``hard phase'' can be thought of as the closest analogue of a relativistic fluid to a classical \emph{incompressible} fluid.\footnote{This analogue has guided the earliest investigations of this model, see~\cite{Buchdahl:68}.} When surrounded by vacuum these bodies have a boundary where the pressure vanishes and the density has a discontinuity.

Presently the local well-posedness of the resulting Einstein-Euler free-boundary problem is an \emph{open} question, at least in the absence of any symmeries.\footnote{A priori estimates were recently obtained by \textsc{Oliynyk} \cite{Oliynyk:17}, and \textsc{Ginsberg} \cite{Ginsberg:18}. The corresponding \emph{non-relativistic} problem, namely the well-posedness of an \emph{incompressible} fluid with free boundary, was solved by \textsc{Lindblad} \cite{Lindblad:05}.}  However in \emph{spherical symmetry} the relevant existence and uniqueness results have been obtained by \textsc{Kind-Ehlers} \cite{Kind:93}, and \textsc{Christodoulou} \cite{Ch:96:II}, who also addressed much more difficult questions related to \emph{gravitational collapse}, namely the continuation and termination of the boundary in the context of a genuine two phase model with \emph{phase transitions} \cite{Ch:96:III, Ch:16:IV}.

In this note we show \textbf{(in Section~\ref{sec:existence})} that this model permits a 1-parameter family of non-trivial \emph{time-independent} solutions, which represent stars in hydrostatic equilibrium. We conjecture that these ``solitons,''  when \emph{``small,''} are dynamically stable, namely that under small perturbations of the initial data these stars neither ``collapse under their own weight,'' nor ``disperse;''  see Conjecture~\ref{conj:stability} below.\footnote{Another reason for a possible breakdown of the solution is the \emph{formation of shocks} \cite{Ch:07}. However in this model the characteristics of the fluid coincide with the null geodesics of the metric, and are thus \emph{fixed}. While in spherical symmetry this prevents the characteristics from crossing it is still plausible that they accumulate after a number of reflections at the boundary.} In spherical symmetry this is a subtle issue,  because \emph{there is no mechanism by which internal energy can be radiated away}: the exterior of the star is a vacuum region, hence always isometric to a Schwarzschild solution (with a fixed mass), and has no gravitational degrees of freedom.

Indeed, based on these considerations one could expect quite the opposite: The physical boundary condition of vanishing pressure induces essentially a reflecting boundary condition for the internal oscillations, and the system then bares some resemblance to the situation in \textsl{Anti-de Sitter (AdS) with reflecting boundary conditions}, for which \textsc{Moschidis} recently showed an \emph{instability}, namely the formation of black holes from arbitrarily small perturbations of AdS data (in the context of the spherically symmetric Einstein--massless Vlasov system) \cite{Moschidis:17, Moschidis:18}; see also the numerical studies for the Einstein--scalar field system, initiated by \textsc{Bizo\'n} and \textsc{Rostworowski} \cite{Bizon:11}.

The property that supports the idea that hard stars are stable, is then the following: We show \textbf{(in Section~\ref{sec:variation})} that the equilibrium configuration of a \emph{small} star lies in a \emph{local minimum of the mass energy functional for fixed total particle number}. Both quantities, the total energy and the number of particles are \emph{conserved} in the evolution, but the norm that is controlled by the second variation is not strong enough to provide an immediate \emph{orbital stability} result following \textsc{Grillakis-Shatah-Strauss} \cite{Grillakis:90}.\footnote{An example of a self-gravitating soliton which was proven to be orbitally stable in spherical symmetry, with a related approach due to \textsc{Cazenave} and \textsc{Lions}, are the ``galactic clusters'' in the context of the classical Vlasov-Poisson system studied in \cite{Raphael:11, Raphael:12}.} In fact, the second variation just \emph{fails} to give control on the relevant quantities in \emph{bounded variation}, and thus lies below the critical norm for the existence theory for the Einstein-scalar field system in spherical symmetry \cite{Ch:93}. However, we prove \textbf{(in Section~\ref{sec:variation:second})} that the second variation controls (pointwise) the so-called ``mass aspect'' function at the center, and thus gives control precisely of the quantity relevant to a \emph{continuation criterion} for this system \cite{Ch:93}.\footnote{We elaborate below, and in Section~\ref{sec:existence} on the relation to the Einstein-scalar field model.} This gives hope that despite the weakness of the norm a full non-linear stability result can still be obtained by a continuation argument.

Finally \textbf{(in Section~\ref{sec:linear})} we study the linearized system in spherical symmetry and prove a uniform boundedness result for an energy which turns out to be \emph{precisely} at the level of norm controlled by the second variation of the static solutions. Moreover, we exhibit the existence of time periodic solutions, which shows that boundedness (as opposed to decay) of the energy is indeed optimal. We emphasize that this is achieved by a reduction of the system to a ``master equation'' for a single quantity. We expect that these linear oscillations can be integrated to ``non-linear breathers,''\footnote{In a related setting, the integration of linear to ``non-linear hair'' was carried out by \textsc{Chodosh} and \textsc{Shlapentokh-Rothman}, who showed the existence of periodic solutions to the Einstein--Klein-Gordon equations \cite{Chodosh:15}. See also \cite{Kichen:08}.} which would show that an \emph{orbital} stability (as opposed to \emph{asymptotic} stability) result is in fact optimal.


\subsection{Description of the fluid model}

The matter model considered in this note is a \emph{perfect fluid} whose energy-momentum tensor is given by
\begin{equation}
  T=\rho\: u\otimes u+p\,(g+u\otimes u)
\end{equation}
where $\rho$ is the energy density, $u$ the fluid velocity (a \emph{unit} time-like vectorfield), and $p$ the pressure.
Moreover $g$ is the spacetime metric, which is \emph{an unknown} coupled to the fluid variables via the Einstein equations, schematically denoted by $E(g)=8\pi T$, where $T$ enters on the right hand side; see \cite{Ch:95:I}. By virtue of the Bianchi identities, the Einstein equations imply the conservation laws
\begin{equation}\label{eq:nabla:T}
  \nabla\cdot T=0
\end{equation}
which are the \emph{relativistic Euler equations}. This system is not closed, unless supplemented by a thermodynamic equation of state. For \emph{barotropic} fluids the missing equation is provided by a direct functional dependence of the  pressure on the density,
\begin{equation}\label{eq:f}
  p=f(\rho)
\end{equation}
and if $f$ is strictly increasing, one defines
\begin{equation}\label{eq:F}
 F=\int_0^p\frac{\ud p}{f^{-1}(p)+p}
\end{equation}
The Euler equations can then be expressed in terms of the future-directed time-like vectorfield
\begin{equation}
  V=e^F u\,;
\end{equation}
see e.g.~(1.13a) and (1.13c) in \cite{Ch:95:I}.


For an \emph{irrotational} fluid $V$ is the gradient of a function $\phi$,
\begin{equation}\label{eq:V:phi}
  V^\mu=-g^{\mu\nu}\partial_\nu \phi
\end{equation}
and the Euler equations \eqref{eq:nabla:T} reduce, in general, to a \emph{non-linear} wave equation of the form
\begin{equation}\label{eq:wave:G}
  \nabla^\mu(G(\lVert \ud \phi \rVert)\partial_\mu \phi)=0
\end{equation}
where $G(\sigma)=(\rho+p)/\sigma^2$.
Note since $V$ is time-like, $\phi$ is always increasing towards the future, hence we can think of $\phi$ as  a \emph{time function}.

In \textsc{Christodoulou}'s two phase-model the fluid is in a \emph{hard phase} if the density is above a critical density $\rho_0$, and in a \emph{soft phase} if the density falls below it. The former is characterised by a \emph{linear} relationship of the pressure and the density in \eqref{eq:f}, and in the latter the pressure vanishes identically.\footnote{The ``soft phase'' thus coincides with the so-called \emph{dust} model, which by itself is an insufficient model for gravitational collapse \cite{Ch:84}. More recently, also an ``elastic'' model has been considered, which consists only of the ``hard phase'', but allows the pressure to become \emph{negative}; see \cite{costa:18}.}
In summary,

\begin{description}
\item[Two-phase model:] We consider a barotropic fluid which satisfies
  \begin{equation} \label{eq:model}
  p=
  \begin{cases}
    \rho-\rho_0 & \rho\geq \rho_0 \\
    0 & \rho<\rho_0
  \end{cases}
\end{equation}

\end{description}

In this note however, all solutions have the property that $\rho\geq \rho_0$. (The exterior of the star can be viewed as a trivial solution to the soft phase with $\rho=0$.)
Moreover, by a choice of units we can take 
\begin{equation}\label{eq:units}
  \rho_0=1\,.
\end{equation}

The important implication of the choice \eqref{eq:model}, for $\rho>\rho_0$, is that the non-linear wave equation \eqref{eq:wave:G} reduces to the \emph{linear wave equation}\footnote{Of course, given that the Einstein-Euler equations are a \emph{system} for $(g,\phi)$, equation \eqref{eq:wave:intro} is not truly linear. However, for \emph{fixed} $g$, this is a linear equation for $\phi$, which of course plays an important role for the linearisation of the system.}
\begin{equation}\label{eq:wave:intro}
  \Box_g \phi=0\,.
\end{equation}

In the resulting theory, which has been formulated and studied in \cite{Ch:95:I,Ch:96:II,Ch:96:III,Ch:16:IV},
the ``hard phase'' is thus similar to the \emph{massless scalar field model in spherical symmetry} (see \cite{Ch:86,Ch:93});
except that the solution $\phi$ is subject to the constraint
\begin{equation}\label{eq:d:phi}
  \lVert \ud \phi \rVert := \sqrt{-g^{\mu\nu}\partial_\mu\phi\partial_\nu \phi}\geq 1
\end{equation}
and that the energy momentum tensor differs as follows:\footnote{In particular the additional term in the energy momentum tensor \emph{breaks the scale invariance} of the system which is crucially exploited in the existence theory for bounded variations solutions in \cite{Ch:93}.}
\begin{equation}
\label{Tmunu}  T_{\mu\nu}=\partial_\mu\phi\,\partial_\nu\phi+\frac{1}{2}\bigl(\lVert \ud \phi\rVert^2-1\bigr)g_{\mu\nu}
\end{equation}

\begin{rem}
  We emphasize that only for the equation of state \eqref{eq:model} the wave equation for $\phi$ is ``linear,'' which is related to the fact that in the hard phase the speed of sound equals the speed of light, $c_s=1$, where \begin{equation}c_s^2=\frac{\ud p}{\ud \rho}\,.\end{equation}
  Indeed, even for an equation of state which maintains linearity
  \begin{equation}\label{eq:model:c}
    \rho=c_s^2(\rho-\rho_0)\,,\qquad 0<c_s<1\,,
  \end{equation}
  but introduces a second characteristic surface, corresponding to a speed of sound which is strictly smaller than the speed of light,
  the corresponding wave equation \eqref{eq:wave:G} is \emph{non-linear}:
  Using \eqref{eq:F} we compute
  \begin{equation}
    G( \| \ud \phi \| ) = \| \ud \phi \|^\frac{1-c_s^2}{c_s^2}\,.
  \end{equation}
  As opposed to \eqref{eq:wave:intro} we expect shocks to develop for solutions to \eqref{eq:wave:G}, and thus while physically more realistic, \eqref{eq:model:c} provides a less tenable model for the stability analysis discussed below.
\end{rem}

We refer the reader to Section~3 of \cite{Ch:95:I} for a formulation of the equations in various coordinate systems, to Section~6 of \cite{Ch:95:I} for the basic theorems regarding the Cauchy problem, and to Section~2 of \cite{Ch:96:II} for the relevant local existence theory of the free-boundary problem.

\subsection{Summary}

In this note we study spherically symmetric solutions $(\mathcal{M},g,\phi)$ to the \textsl{two phase model}, namely the Einstein-Euler equations with a barotropic equation of state \eqref{eq:model}. We are interested in ``compactly supported'' solutions with the following properties:
\begin{enumerate}[label=(\alph*)]
\item In the quotient $\mathcal{Q}=\mathcal{M}/\mathrm{SO}(3)$ there is a time-like curve $\mathcal{B}$ which separates the ``interior'' from the ``exterior''. The interior contains the central geodesic $\Gamma$, the center of symmetry; c.f.~Fig.~\ref{fig:Q}.
\item  In the interior the fluid is irrotational and the velocity potential $\phi$ satisfies \eqref{eq:d:phi}, or equivalently $\rho\geq 1$.
\item On the boundary $\mathcal{B}$ the physical boundary condition $p=0$ is satisfied.
The density has a discontinuity across $\mathcal{B}$: its induced value is $\rho=1$ from the interior, and $\rho=0$ from the exterior.
  \item In the exterior $\rho\equiv 0$, and $g$ is isometric to an exterior domain of the Schwarzschild solution with a fixed mass $M>0$.
\end{enumerate}

\begin{figure}[tb]
  \centering
  \includegraphics{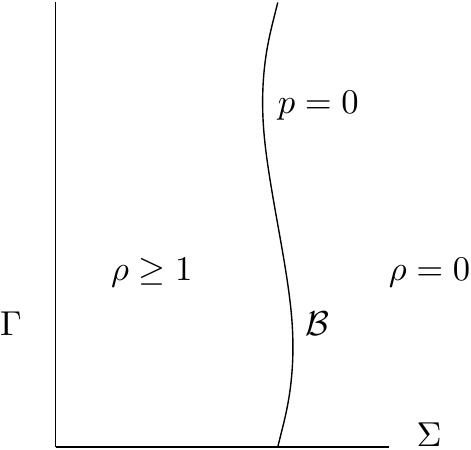}
  \caption[Spherically symmetric stars.]{Interior and exterior of the star seperated by the boundary $\mathcal{B}$ in the quotient manifold $\mathcal{Q}$.}
  \label{fig:Q}
\end{figure}

We demonstrate the existence of a 1-parameter family of \emph{static} solutions with these properties:\footnote{In this paper we call the static solutions ``hard stars.'' However beyond the scope of this paper, a solution to the Einstein-Euler equations with the properties (a-d)  may be called a ``dynamical hard star'', which explains our occasional use of the term ``static hard stars.''}

\begin{description}
\item[Existence of ``\textsl{small} hard stars'' (Prop.~\ref{prop:exist:stars})] {\itshape For any $0<R\ll \sqrt{\frac{3}{4\pi}}$ there exists a \emph{static} solution to the two phase model with the properties~(a-d) above, where $M=\mathcal{O}(R^3)$.}
\end{description}

The proof, which is based on a well-known reduction to the \textsc{Tolman-Oppenheimer-Volkoff} (TOV) equations is carried out in Section~\ref{sec:existence}, where also a more detailed description of their qualitative properties can be found. The earliest discussion of these solutions appears to be given by \textsc{Buchdahl} \cite{Buchdahl:68}.

\begin{rem}
  The above star solutions are ``small'' both in terms of the radius $R$ of the boundary, \emph{and the central density} $\rho\rvert_\Gamma$, in the sense that $0\leq \rho\rvert_\Gamma-1=\mathcal{O}(R^2)$. In the study of static fluid bodies, as initiated by \textsc{Rendall-Schmidt} \cite{Rendall:91} and \textsc{Makino} \cite{Makino:98}, the equilibria are typically parametrized by their central density.
  It is expected that the existence of static stars with \emph{arbitrarily large} central density (yet \emph{finite} extent $R>0$, and total mass $M>0$) can be established, for example by adapting the framework  of \textsc{Ramming-Rein} \cite{Rein:13}, or \textsc{Heinzle-Uggla} \cite{Heinzle:03} to the equation of state \eqref{eq:model}. Their properties would be of great interest with regard to the following conjecture, which is expected to be \emph{false} for stars with \emph{large} central density $\rho\rvert_\Gamma \gg 1$.\footnote{An exciting recent result of \textsc{Had\v zi\'c-Lin-Rein} \cite{Hadzic:18} confirms the scenario that with increasing central density a growing mode instability appears in the linearised Euler-Einstein system around such a star.}
\end{rem}

We expect that in the context of spherical symmetry \emph{small} hard stars are orbitally stable.\footnote{A similar conjecture can of course be formulated outside spherical symmetry. In fact, in the absence of symmetries the emission of gravitational waves may provide an additional stability mechanism on the basis of which one might expect an \emph{asymptotic stability} result. However, not even the \emph{local} existence theory for this problem is available at this stage, and the \emph{a priori} possible formation of shocks may provide another serious obstacle.}

\begin{conj} [Orbital stability of ``\emph{small} hard stars'']
  \label{conj:stability} Consider initial data --- for the Einstein-Euler equations with \eqref{eq:model} ---  which have the properties $(b-d)$ above, and are $\epsilon$-close to a static hard star solution in a suitable norm. Then the solution exists globally in time with the properties $(a-d)$ above and remains $\epsilon$-close to the static star.
\end{conj}

\subsection{Main results }

This paper contains a variational characterisation of small hard stars, and infers an a priori bound on the mass aspect function.
Moreover our analysis of the linearised system shows that the linearised dynamics can be reduced to a ``master equation'' which admits time periodic solutions, and for which the associated energy remains uniformly bounded.

\subsubsection{Variational properties}

In spherical symmetry, an important role is played by the \textsc{Hawking} mass $m$, which is associated to each sphere $q\in\mathcal{Q}$ of radius $r(q)$; see Section~\ref{sec:existence} below, and Section~3 in \cite{Ch:95:I}. It turns out that by virtue of the boundary condition $(c)$ above the Hawking mass is \emph{constant} on the boundary $\mathcal{B}$, and represents the ``total mass energy'' of the star:
\begin{equation}
  M=m\rvert_{\mathcal{B}}
\end{equation}

In fact, in ``comoving coordinates'' $(\phi,\chi)$  the Hawking mass satisfies the equations
\begin{subequations}
  \begin{align}
\frac{\partial m}{\partial \phi}&=-4\pi r^2p\frac{\partial r}{\partial \phi} \label{eq:m:phi}\\
    \label{eq:m:chi}\frac{\partial m}{\partial \chi}&=4\pi r^2 \rho\frac{\partial r}{\partial \chi}
  \end{align}
\end{subequations}
Recall that $\phi$ always satisfies \eqref{eq:d:phi}, and is thus a \emph{time function} whose level sets can be viewed as a curve of ``simultaneous events'' relative to the observers defined by the flow lines of the fluid; the latter are assigned constant values of $\chi$.

Thus at a given time $\phi$ we can view $M$ as a functional of the induced data $(\rho,r)$:
\begin{equation}\label{eq:M:intro}
  M[\rho,r](\phi)= \int_{\Sigma_\phi} \rho \: 4\pi r^2 \ud r
\end{equation}

As already remarked above we have $\partial_\phi M=0$ because $p=0$ on $\mathcal{B}$. Another fundamental conserved quantity is the \emph{total particle number} $N(\phi)$.

We show that the above static solutions can be characterised as follows:
\begin{description}
\item[Variational properties of ``small hard stars''] \textbf{(Prop.~\ref{prop:critical} and Prop.~\ref{moverrprop})}
  {\itshape \emph{Small} hard stars are local minimizers of the mass energy functional $M$ under variations of the data which preserve the total number of particles $N$. Moreover the second variation of the mass functional controls the variation of the mass aspect function:}
    \begin{equation*}
      \lvert \,\Bigl(\frac{m}{r}\Bigr)\dot{} \,\rvert \leq  C r^\frac{1}{2}\:\ddot{M}
    \end{equation*}
\end{description}

More precise statements are given in Section~\ref{sec:variation}. Variational characterisations of ``stars in hydrostatic equilibrium in general relativity'' are not new, and were first given by \textsc{Harrison, Thorne, Wakano,  and Wheeler}; see Chapter~3 in \cite{Thorne:65}.
The novel observation here is the stated inequality whose significance mainly lies in its relevance to the continuation criterion proven in \cite{Ch:93}. It is precisely in the proposed ``hard phase'' model that the Einstein-Euler equations reduce in the irrotational case to a system of equations which is formally similar to the Einstein -- scalar field system. For the latter \textsc{Christodoulou} showed in \cite{Ch:93} that if \emph{the mass aspect $m/r$ is small} at the center, then $\mathrm{C}^1$ solutions to the Einstein -- scalar field system can be extended locally; see Section~5 therein.


\subsubsection{Linearised system}

While the Einstein-Euler equations in spherical symmetry form a \emph{system of equations} (consisting of the Hessian equations for $r$, and the conservation laws \eqref{eq:nabla:T}, see e.g.~Section~\ref{sec:existence} below) we show \textbf{(in Section~\ref{sec:linear:equations})} that essentially due to direct link between the linearised mass $\dot{m}$ and radius $\dot{r}$ in the neighbourhood of small stars (c.f.~Corollary~\ref{cor:dot:m}) the system can be reduced to a single ``master equation'' for $\dot{r}$.

\begin{description}
\item[Decoupled linearised equations (Prop~\ref{prop:mastereq} and Lemma~\ref{lemma:H})]
  {\itshape Consider the linearisation of the Einstein-Euler equations in comoving coordinates $(\phi,\chi)$ around a small star solution $(r_0,\rho_0)$, for fixed particle number N. Then the linearised radius $\dot{r}$ satisfies a \emph{decoupled} wave equation with \emph{mixed boundary conditions} of the form
    \begin{equation*}
      -\partial_\phi^2\dot{r}=H\dot{r}\,,\qquad \partial_\chi \dot{r}\rvert_{\mathcal{B}}=f\,\dot{r}\rvert_{\mathcal{B}}\,,
    \end{equation*}
    where $H$ is a perturbation of $H_0:=-\partial_{r_0}^2-(2/r_0)\partial_{r_0}+2/r_0^2$.
    Moreover, all other linearised quantities can be inferred from explicit formulas relating them to $\dot{r}$.
  }
\end{description}

The structure of the decoupled equation allows us to derive directly a uniform boundedness statement for the associated energy \textbf{(Section~\ref{sec:linear:stability})}. It turns out that the natural energy that arises in the context of the master equation for $\dot{r}$ is exactly at the same level as the norm controlled by the variational argument discussed above. Thus we give two independent proofs of the uniform boundedness statement, and in particular of the pointwise control of the mass aspect function.

\begin{description}
\item[Linear \emph{orbital} stability of small hard stars (Prop.~\ref{prop:linenest}, \ref{prop:highenest})]
  {\itshape Consider a solution $(\dot{r},\dot{\rho})$ to the linearised Einstein-Euler system in comoving coordinates $(\phi,\chi)$ around a small star $(r_0,\rho_0)$.
    If the initial energy $\mathcal{E}(0)$ is finite, then the energy
    \begin{equation*}
      \mathcal{E}(\phi):=\int_0^B \biggl[ \Bigl(\frac{\dot{r}}{r_0}\Bigr)^2+\Bigl(\frac{\partial \dot{r}}{\partial \phi}\Bigr)^2+r_0^4\Bigl(\frac{\partial\dot{r}}{\partial\chi}\Bigr)^2\biggr] \ud \chi
    \end{equation*}
    remains uniformly bounded in time, for all  $\phi\geq 0$.
    Similar uniform boundedness statements hold for all other linearised quantities, and higher order energies.  }
\end{description}

Finally the \emph{linearised} Einstein-Euler system admits time periodic solutions which are $\mathrm{C}^1$-regular at the center $\Gamma$.
This is true because the decoupled wave equation for $\dot{r}$ admits time periodic solutions which satisfy a suitable boundary condition at the center $r_0=0$. This shows that the above uniform boundedness statement is \emph{optimal}.

\begin{description}
\item[Existence of time periodic solutions (Prop.~\ref{prop:periodsol})]
  {\itshape The decoupled equation for $\dot{r}$ admits solutions of the form $\dot{r}=e^{i\sqrt{\lambda_j} \phi}h(r_0)$ which are \emph{regular} at the center and satisfy the vanishing pressure condition at the boundary, where $\lambda_j\to\infty$ is a series of \emph{real} numbers, and $\sqrt{\lambda_1}\sim r_0\rvert_{\mathcal{B}}^{-1}$.}
\end{description}

The existence of \emph{non-linear} periodic solutions to the Einstein-Euler equations in a neighborhood of small hard stars remains a difficult open problem.\footnote{Approximate time periodic solutions to the Euler-Poisson system near Lane-Emden stars have recently been constructed in \cite{Jang:16}.}


\begin{quote}
\textbf{Acknowledgements.} {\small We would like to thank Robert Wald for a useful discussion at the trimester program on ``Mathematical General Relativity'' at the IHP, in Paris (2015). G.F.~was supported by the \texttt{EPSRC grant EP/K00865X/1} on ``Singularities of Geometric Partial Differential Equations'' and partially by the \texttt{ERC grant 714408 GEOWAKI} under the European Union's Horizon 2020 research and innovation program. V.S.~gratefully acknowledges the supported of \texttt{ERC consolidator Grant 725589 EPGR}, and \texttt{ERC advanced grant 291214 BLOWDISOL}.}
\end{quote}

\section{Existence of hard stars and their properties}
\label{sec:existence}


In spherical symmetry the metric takes the form
\begin{equation}
  g=-\Omega^2\ud u\ud v+r^2\mathring{\gamma}
\end{equation}
where $(u,v)$ are double null coordinates on $\mathcal{Q}$. Moreover the \textsc{Hawking} mass is defined by:
\begin{equation}\label{eq:m}
  1-\frac{2m}{r}=-\frac{4}{\Omega^2}\frac{\partial r}{\partial u}\frac{\partial r}{\partial v}
\end{equation}

Recall from \cite{Ch:95:I} that with the notation
\begin{equation}\label{eq:sigma}
  \sigma=\lVert \ud \phi \rVert
\end{equation}
we have the relation
\begin{equation}
  2\rho-1=\sigma^2 \label{eq:rho:sigma}
\end{equation}
in the ``hard phase.''

Moreover in null coordinates, the equation \eqref{eq:wave:intro} can be expressed as
\begin{equation}\label{eq:wave}
  \frac{\partial^2\phi}{\partial u\partial v}+\frac{1}{r}\frac{\partial r}{\partial u}\frac{\partial \phi}{\partial v}+\frac{1}{r}\frac{\partial r}{\partial v}\frac{\partial \phi}{\partial u}=0
\end{equation}
and the Hessian equations for the radius function read, c.f.~(3.48) in \cite{Ch:95:I}:
\begin{subequations}\label{eq:hessian}
  \begin{align}
    \frac{\partial^2 r}{\partial u^2}-\frac{2}{\Omega}\frac{\partial \Omega}{\partial u}\frac{\partial r}{\partial u}&=-4\pi r\bigl(\frac{\partial \phi}{\partial u}\bigr)^2\\
    \frac{\partial^2 r}{\partial u\partial v}+\frac{1}{r}\frac{\partial r}{\partial u}\frac{\partial r}{\partial v}&=\frac{\Omega^2}{4r}\bigl(4\pi r^2-1\bigr) \label{eq:hessian:uv}\\
    \frac{\partial^2 r}{\partial v^2}-\frac{2}{\Omega}\frac{\partial \Omega}{\partial v}\frac{\partial r}{\partial v}&=-4\pi r\bigl(\frac{\partial \phi}{\partial v}\bigr)^2
  \end{align}
\end{subequations}

For future reference we also note the mass equations, c.f.~(3.50) in \cite{Ch:95:I},
\begin{subequations}\label{eq:mass}
  \begin{align}
    \frac{\partial r}{\partial u}\frac{\partial m}{\partial u} &= 2\pi r^2\Bigl[\Bigl(\frac{\partial r}{\partial u}\Bigr)^2+\Bigl(1-\frac{2m}{r}\Bigr)\Bigl(\frac{\partial \phi}{\partial u}\Bigr)^2\Bigr]\label{eq:mass:u}\\
    \frac{\partial r}{\partial v}\frac{\partial m}{\partial v} &= 2\pi r^2\Bigl[\Bigl(\frac{\partial r}{\partial v}\Bigr)^2+\Bigl(1-\frac{2m}{r}\Bigr)\Bigl(\frac{\partial \phi}{\partial v}\Bigr)^2\Bigr]
  \end{align}
\end{subequations}

\begin{rem}
  The above system is \emph{not} identical to the Einstein -- scalar field system. Note in particular the additional term on the right hand side of \eqref{eq:hessian:uv}, which breaks the scale invariance of the system; c.f.~(1.4a) in \cite{Ch:93}, and Section~2 therein. 
\end{rem}

\subsection{Derivation of the hydrostatic equations in the hard phase}

Here we are interested in \emph{static} solutions,
for which the fluid velocity
\begin{equation}
   V^\mu=\sigma u^\mu=-g^{\mu\nu}\partial_\nu \phi
\end{equation}
generates an isometry, (and is orthogonal to the level sets of $\phi$).

Let the quotient $\mathcal{Q}$ be covered by null coordinates $(u,v)$ such that $u=-v$ on the initial hypersurface $\phi=0$, and $u=v$ at the center. Then the fluid velocity potential $\phi$ is a function of $(u,v)$, and 
we make the ansatz
\begin{equation}\label{eq:star:phi:ansatz}
  \phi=u+v
\end{equation}
so that
\begin{equation}
  V=\frac{2}{\Omega^2}\Bigl(\frac{\partial}{\partial u}+\frac{\partial}{\partial v}\Bigr)
\end{equation}

For the solution we derive all fluid variables, including $\phi$ and $V$, are in fact compactly supported in $v-u$. We will discuss the yet to be determined boundary below, and in the surrounding vacuum region the spacetime is always isometric to Schwarzschild.

With $\phi$ given by \eqref{eq:star:phi:ansatz} we have
\begin{equation}
  \lVert \ud\phi\rVert^2=\frac{4}{\Omega^2}
\end{equation}
which implies in view of \eqref{eq:sigma} and \eqref{eq:rho:sigma} that
\begin{equation}\label{eq:rho:Omega}
  2\rho-1=\frac{4}{\Omega^2}\,.
\end{equation}

Now the wave equation \eqref{eq:wave} implies
\begin{equation}
  \frac{1}{r}\frac{\partial r}{\partial u}+\frac{1}{r}\frac{\partial r}{\partial v}=0\ \text{, or}\quad V\cdot r=0
\end{equation}
which says that $r$ is just a function of
\begin{equation}
  x=v-u
\end{equation}
(we will denote differentiation by $x$ by $'$).
Substituting into the mass equations \eqref{eq:mass} further implies that
\begin{equation}
  V\cdot m=0
\end{equation}
hence also $m=m(x)$.

The Hessian equation \eqref{eq:hessian:uv} now implies 
\begin{equation}
  -r''-\frac{1}{r}r'r'=\frac{\Omega^2}{4r}\bigl(4\pi r^2-1\bigr)
\end{equation}
and from \eqref{eq:mass:u} we infer
\begin{equation}
  r'm'=2\pi r^2\Bigl[(r')^2+\bigl(1-\frac{2m}{r}\bigr)\Bigr]
\end{equation}
Moreover, recall the defining equation for the mass function \eqref{eq:m} which tells us
\begin{equation}\label{eq:m:def}
  1-\frac{2m}{r}=\frac{4}{\Omega^2}(r')^2
\end{equation}
These two equations together give us the o.d.e for $m(r)$:
\begin{equation}
      \frac{\ud m}{\ud r}=\frac{m'}{r'}=2\pi r^2\Bigl[1+\frac{4}{\Omega^2}\Bigr]
\end{equation}
or simply
\begin{equation}
    \frac{\ud m}{\ud r}=4\pi r^2 \rho\label{eq:m:ode}
\end{equation}

Next we have to derive an equation for $\rho$. Substituting in \eqref{eq:rho:Omega} for $\Omega^2$ from \eqref{eq:m:def}, and differentiating  with respect to $x$, we find
\begin{subequations}
  \begin{gather}
    2\rho=1+\frac{1}{(r')^2}\Bigl(1-\frac{2m}{r}\Bigr)\\
    2\rho'=-\frac{1}{r'}\frac{2}{r}\Bigl[4\pi r^2\bigl(\rho-1\bigr)+\frac{m}{r}\Bigr]
  \end{gather}
\end{subequations}
  or simply
\begin{equation}
      \frac{\ud \rho}{\ud r}=-\frac{2\rho-1}{r-2m}\Bigl[4\pi r^2(\rho-1)+\frac{m}{r}\Bigr]\label{eq:rho:ode}
\end{equation}
The latter is a special case of the \textsc{Tolman-Oppenheimer-Volkoff} equation \cite{Oppenheimer:39}.

Note in particular that
\begin{equation}
  \frac{\ud \rho}{\ud r}\leq 0
\end{equation}
because $\rho\geq 1$ in the hard phase, and $r>2m$ in absence of trapped surfaces.
Solutions are now obtained by solving the o.d.e.'s \eqref{eq:m:ode}, and \eqref{eq:rho:ode} from the center $r=0$ to the boundary where
\begin{equation}
  p=0\,,\qquad \rho=1\,.
\end{equation}

\subsection{Existence of stars}
\label{sec:exist:stars}

The system of o.d.e.'s for $m$, and $\rho$, with the boundary values $m(0)=0$ and $\rho(R)=1$ has a continuous solution for all $R>0$ sufficiently small, which shows the existence of a family of \emph{small static stars} in the two phase model, parametrised by the radius of the boundary $R$.

\begin{prop} \label{prop:exist:stars}
  For any $0<R<R_0$, $R_0\ll \sqrt{\frac{3}{4\pi}}$, there exists a continuous solution $(m_R,\rho_R)$ to the system of o.d.e.'s (\ref{eq:m:ode},\ref{eq:rho:ode}), $\rho$ monotone decreasing, and $m$ monotone increasing as a function of $r\in[0,R]$, with the properties that for some $C>0$, independent of $R$, 
  \begin{align*}
    1\leq\rho_R(r)\leq 1+C R^2  &\qquad \rho_R(R)=1 \\
     m_R(0)=0 &\qquad 1\leq\frac{3}{4\pi}\frac{m_R(r)}{r^3}\leq 1+C R^2
  \end{align*}

\end{prop}

\begin{rem}
  We include here for completeness the existence proof for \emph{small} stars, which is based on a simple contraction mapping, and suffices for the analysis of the linearised equations in Section~\ref{sec:linear}.
  Alternatively, this could be approached using the formulation \cite{Rein:13}, which has the advantage that their proof likely extends to the \emph{large central density} regime; see also \cite{Heinzle:03} for a dynamical systems formulation. The analysis of hard stars with large central densities lies outside the scope of this paper, but given their drastically different properties, their further investigation is an interesting topic; c.f.~\cite{Makino:98}.
\end{rem}

We write
\begin{subequations}
\begin{align}
  m&=\frac{4\pi}{3} r^3+\tilde{m}\\
  \rho&=1+\tilde{\rho}
\end{align}
\end{subequations}
which is motivated by an approximation of $\rho$ by its value on the boundary, where $\rho=1$.
Then $\mt$, and $\rhot$ satisfy,
\begin{subequations}\label{eq:static:tilde}
\begin{align}
  \frac{\ud \mt }{\ud r}&=4\pi \rhot r^2\\
  \frac{\ud \rhot}{\ud r}&=-\frac{1+2\rhot}{1-\frac{8\pi}{3}r^2-\frac{2\mt}{r}}\frac{4\pi r}{3}\Bigl[3\rhot+1+\frac{3\mt}{4\pi r^3}\Bigr]\label{eq:static:rho:tilde}
\end{align}
\end{subequations}
with the boundary values:
\begin{equation}\label{eq:static:tilde:boundary}
  \mt(0)=0\qquad \rhot(R)=0
\end{equation}


We will give a standard proof of the existence of solutions to this system of o.d.e.'s by invoking the contraction mapping principle in the space:
\begin{multline}
  \mathcal{C}_R^0:= \Bigl\{ [0,R] \longrightarrow \mathbb{R}^2 \text{ continuous}, r\mapsto (m,\rho)  :\\ 0\leq m\leq \frac{4\pi }{3}\rho_0(R) r^3, 0\leq \rho\leq \rho_0(R)\Bigr\}
\end{multline}
where
\begin{equation}
  \rho_0(R) = \frac{16\pi}{3} R^2
\end{equation}

For this purpose define
\begin{subequations}
\begin{align}
  T[m,\rho]&=(M[\rho],P[m,\rho])\\
              M[\rho](r)&= \int_0^r 4\pi r^2 \rho(r) \ud r\\
              P[m,\rho](r)&= \int_r^R\frac{1+2\rho}{1-\frac{8\pi}{3}r^2-\frac{2m}{r}}\frac{4\pi r}{3}\Bigl[1+3\rho+\frac{3m}{4\pi r^3}\Bigr]             
\end{align}
\end{subequations}

First we show that $T$ is a bounded map on $\mathcal{C}_R^0$ endowed with the norm
\begin{equation}
  \lVert (m,\rho) \rVert := \frac{3}{4\pi}\lVert \frac{m}{r^3} \rVert_\infty + 2 \lVert \rho \rVert_\infty
\end{equation}

\begin{lemma} For all $0<R<\frac{1}{2}\sqrt{\frac{3}{8\pi}}$, we have  $T:\mathcal{C}_R^0\to\mathcal{C}_R^0$.
\end{lemma}
\begin{proof}
  Suppose $(m,\rho)\in\mathcal{C}_R^0$,  then $M[\rho]\geq 0$, and
  \begin{equation*}
    M[\rho]\leq \frac{4\pi}{3}\rho_0 r^3\,.
  \end{equation*}

  Moreover,
  \begin{equation*}
    1-\frac{8\pi}{3}r^2-\frac{2m}{r}\geq 1-\frac{8\pi}{3}(\rho_0(R)+1)R^2\geq \frac{1}{2}
  \end{equation*}
  and therefore $P[m,\rho]\geq 0$ is decreasing in $r$, and
  \begin{equation*}
    \begin{split}
      P[m,\rho]&\leq 2(1+2\rho_0)\int_r^R\frac{4\pi r}{3}(1+4\rho_0)\ud r\leq \frac{4\pi}{3}(1+4\rho_0(R))^2\bigl(R^2-r^2\bigr)\\
      &\leq \frac{16\pi}{3}R^2
    \end{split}
  \end{equation*}
because with the assumed bound on $R$, and choice of $\rho_0(R)$: $\rho_0\leq \frac{1}{2}$.

\end{proof}

Second we show the contraction property:

\begin{lemma} For $0<R<\frac{1}{2}\sqrt{\frac{3}{4\pi}}$ sufficiently small, $T$ is a contraction on $(\mathcal{C}_R^0,\lVert \cdot \rVert)$,
  \begin{equation*}
    \lVert T[m_1,\rho_1]-T[m_2,\rho_2] \rVert \leq \frac{3}{4} \lVert (m_1-m_2,\rho_1-\rho_2) \rVert
  \end{equation*}
  
\end{lemma}

\begin{proof}
  For short let $M_i=M[\rho_i]$, $i=1,2$, and $P_i=P[m_i,\rho_i]$, $i=1,2$. Then
  \begin{equation*}
    M_1(r)-M_2(r)=\int_0^r 4\pi r^2(\rho_1-\rho_2)\ud r\leq \frac{4\pi}{3}\lVert \rho_1-\rho_2\rVert_\infty r^3
  \end{equation*}
  hence
  \begin{equation*}
    \frac{3}{4\pi}\lVert \frac{M_1-M_2}{r^3} \rVert_\infty \leq \lVert \rho_1-\rho_2 \rVert_\infty\,.
  \end{equation*}

  Moreover
  \begin{align*}
    P_1(r)-P_2(r) &= \int_r^R \frac{4\pi r}{3}\Pi_{i=1}^2\Bigl(1-\frac{8\pi}{3}r^2-\frac{2m_i}{r}\Bigr)^{-1}\times\\
                  &\times\biggl[\Bigl(1-\frac{8\pi}{3}r^2\Bigr)\Bigl(5(\rho_1-\rho_2)+\frac{3}{4\pi r^3}(m_1-m_2)+6(\rho_1^2-\rho_2^2)\\
                  &\qquad +6(\rho_1-\rho_2)\frac{m_1}{4\pi r^3}+6\rho_2\bigl(\frac{m_1}{4\pi r^3}-\frac{m_2}{4\pi r^3}\bigr)\Bigr)\displaybreak[0]\\
                  &-\frac{2(m_2-m_1)}{r}-\frac{2m_2}{r}(2\rho_1-2\rho_2)+\frac{2(m_1-m_2)}{r}2\rho_2\\
                  &-\frac{2m_2}{r}\Bigl(3\rho_1+\frac{3m_1}{4\pi r^3}\Bigr)+\frac{2m_1}{r}\Bigl(3\rho_2+\frac{3m_2}{4\pi r^3}\Bigr)\\
                  &-\frac{2m_2}{r}2\rho_1\Bigl(3\rho_1+\frac{3m_1}{4\pi r^3}\Bigr)+\frac{2m_1}{r}2\rho_2\Bigl(3\rho_2+\frac{3m_2}{4\pi r^3}\Bigr)\biggr]
  \end{align*}
hence
\begin{equation*}
  2 \lVert P_1-P_2 \rVert_\infty \leq C R^2 \Bigl[ 2 \lVert \rho_1-\rho_2 \rVert_\infty + \frac{3}{4\pi} \lVert \frac{m_1-m_2}{r^3} \rVert_\infty\Bigr]
\end{equation*}
for some numerical constant $C>0$,
where we used that 
\begin{equation*}
  2 \lVert \rho_i \rVert_\infty + \frac{3}{4\pi} \lVert \frac{m_i}{r^3} \rVert_\infty \leq 4\rho_0(R)\leq  2
\end{equation*}
Since for $R$ sufficiently small, $C R^2\leq\frac{1}{4}$, the statement follows.

\end{proof}

This yields by the Banach fixed point theorem the existence of $(\mt,\rhot)\in \mathcal{C}_R^0$, for any $R>0$ sufficiently small, such that
\begin{equation}
  M[\mt]=\mt\qquad P[\mt,\rhot]=\rhot
\end{equation}
namely a family of solutions $(\mt_R,\rhot_R)$ to \eqref{eq:static:tilde}, satisfying the boundary conditions \eqref{eq:static:tilde:boundary} parametrized by $0<R\ll \sqrt{\frac{3}{4\pi}}$.
This completes the proof of Proposition~\ref{prop:exist:stars}.

\subsection{Discussion of the solutions}

While the properties of hard stars given in the existence proof in Section~\ref{sec:exist:stars} are sufficient to proceed with the variational characterisation in Section~\ref{sec:variation}, it is nonetheless interesting to point out a number of qualitative features of these solutions.

\subsubsection{Remarks on the size of the stars}

In this section we include the natural constants $G$, and $c$, and recall that $[G/c^2]=\text{Length}/\text{Mass}$.

\begin{rem}
 A ``regular'' hard star is no larger than
\begin{equation}
  R<R_\circ := \sqrt{\frac{c^2}{G}\frac{3}{8\pi}\frac{1}{\rho_0}} \,.
\end{equation}

\end{rem}

Indeed, on one hand since $\rho\geq \rho_0$, we have
\begin{equation}\label{eq:m:lower}
  m\geq \frac{4\pi}{3} \rho_0 r^3
\end{equation}
On the other hand, if we  require that none of the spherical shells is trapped -- this is what we mean by ``regular'' -- then
\begin{equation}
  \frac{2 G m}{c^2 r}<1\,,
\end{equation}
and we obtain the above \emph{upper bound on the size of the star}.

For typical nuclear densities of say $\rho_0\simeq 2.3\times10^{17} kg/m^3$ we find
\begin{equation}
  R_0\simeq 26\, km\,.
\end{equation}

\begin{rem}
Finally ``small'' stars  have \emph{no photonsphere}. Indeed it follows from the bound on $m$ in Prop.~\ref{prop:exist:stars} that the surface radius satisfies 
\begin{equation}
  R \geq \frac{3}{4\pi}\frac{1}{1+CR^2}\frac{1}{R^2}m_R(R)>3m(R)
\end{equation}
for $R$ sufficiently small.
\end{rem}

\subsubsection{Qualitative behaviour of the density profile}
\label{sec:star:approximate}

We remark that a qualitative picture of the density profile can be inferred from an approximate o.d.e.~derived below, which formally yields an expansion for small stars near the boundary of the form
\begin{equation}\label{eq:rho:approx}
  \rho\simeq 1+\frac{2\pi}{3}\bigl(R^2-r^2)
\end{equation}
as depicted in Figure~\ref{fig:star:graph}.

In Section~\ref{sec:exist:stars} we have seen $\rho=1$ and $m=(4\pi/3) r^3$ are good approximations of the static solutions in the sense that all deviations are $\mathcal{O}(R^2)$. Now observe that inserting this approximation for $m$ in the equation for $\rhot$, $\rho=1+\rhot$, the o.d.e.~becomes explicitly integrable. Indeed, setting $\mt=0$ in \eqref{eq:static:rho:tilde}, the approximate o.d.e. for $\rhot$ reads
\begin{equation}
    \frac{\ud \rhot}{\ud r}=-\frac{1+2\rhot}{1-\frac{8\pi}{3}r^2}4\pi r\Bigl[\rhot+\frac{1}{3}\Bigr]
\end{equation}
which separates into
\begin{equation}
  \int_{0}^{\rhot}\frac{\ud \rho}{\bigl(\rho+\frac{1}{2}\bigr)\bigl(\rho+\frac{1}{3}\bigr)}=\int_r^R\frac{8\pi r}{1-\frac{8\pi}{3}r^2}\ud r
\end{equation}
where we have taken into account the boundary condition $\rhot(R)=0$.
Thus
\begin{equation}
  6\log\Bigl\lvert\frac{\rho+\frac{1}{3}}{\rho+\frac{1}{2}}\Bigr\rvert^{\rhot}_0=-\frac{3}{2}\log\Bigl\lvert 1-\frac{8\pi}{3}r^2\Bigr\rvert^R_r
\end{equation}
or
\begin{equation}
  \frac{3\rho_1+1}{2\rho_1+1}=\Bigl\lvert\frac{1-\frac{8\pi}{3}r^2}{1-\frac{8\pi}{3}R^2}\Bigr\rvert^\frac{1}{4}
\end{equation}
Since $\rhot=\mathcal{O}(R^2)$ an expansion on both sides is justified for small $R\ll R_0$, which gives:
\begin{equation}\label{eq:rho:approx}
  \rhot(r)\simeq \frac{2\pi}{3}\bigl(R^2-r^2\bigr)
\end{equation}

\begin{figure}[tb]
  \centering
  \includegraphics{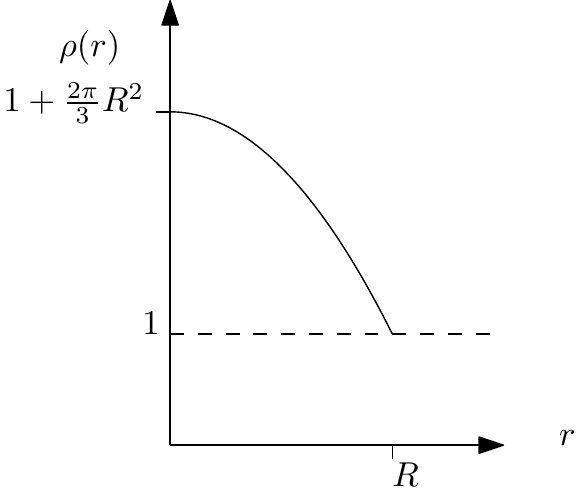}
  \caption{Density distribution of an approximate solution to the o.d.e. system for hard stars.}
  \label{fig:star:graph}
\end{figure}

Moreover from \eqref{eq:rho:Omega} we also obtain
\begin{equation}
  \frac{4}{\Omega^2}\simeq 1+\frac{4\pi}{3}\bigl(R^2-r^2\bigr)
\end{equation}
and  from \eqref{eq:m:def} it now follows that, for small $R\ll R_\circ$:
\begin{equation}
  (r')^2=\frac{\Omega^2}{4}\Bigl(1-\frac{2m}{r}\Bigr)\simeq 1-\frac{4\pi}{3}\bigl(R^2+r^2\bigr)
\end{equation}
or
\begin{equation}\label{eq:r:prime:approx}
  r'\simeq  1-\frac{2\pi}{3}\Bigl(R^2+r^2\Bigr)
\end{equation}
In particular $\frac{1}{2}\leq r'\leq 1$ for all $0\leq r\leq R$.

\begin{rem}
  As we have seen above the static solutions constructed here have finite extend, and the density has a jump at the boundary (Figure~\ref{fig:star:graph}). We point out that this is not the case for all static solutions to the Einstein-Euler equations, but this behaviour depends on the equation of state. See for instance \cite{andersson:18} for solutions with drastically different asymptotic behaviour.
\end{rem}


\section{Variational properties of hard stars}
\label{sec:variation}

In this Section we will characterize the above family of static solutions as local minimizers of the general relativistic mass energy functional.
It is known --- as already described in the book of \textsc{Harrison, Thorne, Wakano, and Wheeler} \cite{Thorne:65} --- that among all ``momentarily static and spherically symmetric configurations which contain a specified number of baryons that configuration which extremizes the mass satisfies the \textsc{Tolman-Oppenheimer-Volkoff} equation of hydrostatic equilibrium.'' The latter is \emph{identical} to \eqref{eq:rho:ode} in the present setting. We follow their approach for the calculation of the first variation, and then proceed to show that the second variation is positive for small hard stars.

As discussed in the introduction, a suitable coordinate system for the variational treatment are the ``comoving-coordinates'' $(\phi,\chi)$ relative to which the metric takes the form
\begin{equation}\label{eq:metric:com}
  g=-e^{2\psi}\ud \phi^2+e^{2\omega}\ud\chi^2+r^2\mathring{\gamma}\,,
\end{equation}
see also Section~3 in \cite{Ch:95:I}. In these coordinates the wave equation \eqref{eq:wave:intro} reduces to the conservation law:
\begin{align}\label{eq:wave:com}
\partial_\phi(r^2e^{-\psi} e^\omega)=0,
\end{align}
where 
\begin{align}\label{sigmacom}
e^{-\psi}=\|d\phi\|=\sigma
\end{align}
By (\ref{eq:rho:sigma}) we also have the relation
\begin{align}\label{psi:rho}
2\rho-1=e^{-2\psi}.
\end{align}

\subsection{First variation}
\label{sec:variation:first}

In short, the idea is to consider variations of the \emph{total mass-energy} 
\begin{equation}
  M=\int \rho \, 4\pi r^2 \ud r
\end{equation}
while keeping the \emph{total number of particles} fixed:
\begin{equation}
  N=\int n \ud \mu 
\end{equation}

A priori $M$ is a functional of  $r$, and $\rho$, which can be expressed with respect to any coordinate $\chi\in [0,B]$:
\begin{equation}\label{def:M}
  M[\rho,r]=\int_0^{B} \rho(\chi) \, 4\pi r^2(\chi) \frac{\partial r}{\partial \chi} \ud \chi
\end{equation}
However, in view of the equation of state, the density is a function of the particle density:
In fact, recall that in the hard phase \cite[1.23a]{Ch:95:I}
\begin{equation}\label{eq:n:sigma}
  \sigma=\frac{1}{v}=n,
\end{equation}
yielding the relation\footnote{here we have set the constant $m_0\equiv 1$ ; c.f.~\cite[1.10b]{Ch:95:I} where $m(s)$ is introduced as a function of the entropy only, which is constant in the present setting.}
\begin{equation}\label{eq:rho:n}
  \rho=\frac{1}{2}\bigl(n^2+1\bigr)
\end{equation}
Moreover, in specific coordinates we can find a formula for $n$ in terms of $r$, and $m$:
We choose $\chi$ to label the spherical shells in such a way that precisely $\chi$ particles are contained in the sphere of radius $r(\chi)$, namely:
  \begin{equation}\label{eq:chi:number}
    \int_0^\chi n\,\ud\mu(\chi)=\chi
  \end{equation}
where
  \begin{equation}
  \ud\mu =4\pi r^2 e^\omega d\chi 
\end{equation}
is the proper volume element. Then we obtain the formula
 \begin{equation}\label{def:n}
   n=\frac{1}{4\pi r^2}e^{-\omega}
 \end{equation}
 where $\omega$ can be determined from the mass equation
 \begin{align}\label{eq:mass:comoving}
    1-\frac{2m}{r}=-e^{-2\psi}\bigl(\frac{\partial r}{\partial \phi}\bigr)^2 +e^{-2\omega}\bigl(\frac{\partial r}{\partial \chi}\bigr)^2,
 \end{align}
 which for static solutions reads:
 \begin{equation}\label{eq:mass:static}
   1-\frac{2m}{r}=e^{-2\omega}\bigl(\frac{\partial r}{\partial \chi}\bigr)^2
 \end{equation}
In this case it follows that $M=M[r]$ is in fact only a functional of $r$.
Note also that in these coordinates $B=N$.
\begin{rem}
The coordinate $\chi$ defined initially via \eqref{eq:chi:number} is extended by the condition $[\partial_\phi,\partial_\chi]=0$. However, it is important to note that the formula \eqref{eq:chi:number} remains valid for all positive $\phi$. Indeed, from \eqref{eq:rho:n}, \eqref{sigmacom} it follows that 
\begin{align}\label{n:psi}
n=e^{-\psi}
\end{align}
Hence, the wave equation (\ref{eq:wave:com}) implies 
\begin{align}\label{dphi:chi}
\partial_\phi(nd\mu)=0\qquad\Rightarrow\qquad\partial_\phi\int^\chi_0nd\mu =0,
\end{align}
which shows that the number of particles enclosed by the sphere of radius $r(\chi,\phi)$ is the same for all $\phi$. 
\end{rem}
\begin{prop}
Let $r(\chi),\rho[r(\chi)]$ be an initial configuration  on $\{\phi=0\}\times[0,B]$, with $\frac{\partial r}{\partial\phi}=0$, for a ``hard star'' with $N$ particles and non-negative pressure which vanishes only at the boundary $p=0$, $\chi=B$.
Then $r$ is a critical point of the mass functional
\begin{equation}
  M[r]=\int_0^{B}\rho(\chi)\, 4\pi r^2 \frac{\partial r}{\partial \chi}\ud \chi
\end{equation}
under variations which preserve the \emph{total number of particles} $N$, 
if and only if the associated density $\rho$ solves the equation
\begin{equation}\label{eq:drhodr}
  \frac{\ud \rho}{\ud r}=-\frac{2\rho-1}{r-2m}\Bigl(\frac{m}{r}+4\pi r^2 (\rho-1)\Bigr)\,.
\end{equation}

\end{prop}
\smallskip

We remark that by assumption $\rho=1+p>1$, $\chi\in[0,B)$ and $r>2m$.\footnote{This is equivalent to the assumption that there  are no trapped shells.} 
Hence, we observe that the density of critical points to the mass functional satisfying \eqref{eq:drhodr} is manifestly decreasing.

By {\it variation} of $r(\chi)$ above we mean a 1-parameter family $r_\lambda,\rho_\lambda,\frac{\partial r_\lambda}{\partial\phi}$ of initial configurations, with $\lambda\in[-1,1]$, $r_0=r(\chi),\rho_0=\rho[r(\chi)]$, such that $\frac{\partial r_0}{\partial\phi}=0$.
Denote by
\begin{equation}
  \dot{M}:=\frac{\ud}{\ud \lambda} M[r_\lambda]\rvert_{\lambda=0}
\end{equation}
the first variation of the total mass and similarly for all other quantities, such as $\dot{\rho}$; in particular we set
\begin{equation}
  \dot{r}:=\frac{\ud}{\ud \lambda}r_\lambda(\chi)\rvert_{\lambda=0}.
\end{equation}
Note that $\chi$ and $\lambda$ are \emph{independent}, i.e., 
 $[\frac{\partial}{\partial \lambda},\frac{\partial}{\partial \chi}]=0$.
Also, this construction implies that all functions are defined on the fixed interval $[0,N]$ and it is here where the constraint on the total particle number is crucially used. We suppress the subscript $\lambda$ below for covenience.
\begin{proof}
Let $n$ be the number density of the given configuration. Choose the coordinate $\chi$ along $\phi=0$ according to (\ref{eq:chi:number}). The formulas \eqref{def:n}, \eqref{eq:mass:comoving}, \eqref{psi:rho} then give
\begin{equation}\label{eq:n:chi}
  n(\chi)=\frac{\sqrt{1-\frac{2m}{r}+(2\rho-1)(\frac{\partial r}{\partial\phi})^2}}{4\pi r^2\frac{\partial r}{\partial\chi}},
\end{equation}
where by assumption $\frac{\partial r}{\partial\chi }>0$. Since the background configuration is static, the first variation of $n$ reads
\begin{equation}\label{eq:ndot}
  \dot{n}=\frac{1}{\sqrt{1-\frac{2m}{r}}}\frac{-\frac{\dot{m}}{r}+\frac{m}{r^2}\dot{r}}{4\pi r^2 \frac{\partial r}{\partial \chi}}-2\frac{\sqrt{1-\frac{2m}{r}}}{4\pi r^3 \frac{\partial r}{\partial \chi}}\dot{r}-\frac{\sqrt{1-\frac{2m}{r}}}{4\pi r^2\bigl( \frac{\partial r}{\partial \chi} \bigr)^2}\frac{\partial \dot{r}}{\partial \chi}
\end{equation}
To avoid confusion we note that all the `undotted' functions in first variation formulas are evaluated at $\lambda=0$.
We compute $\dot{m}$ by taking the variation of  the equation:
  \begin{subequations}
  \begin{gather}
    \frac{\partial  m}{\partial \chi}=4\pi r^2 \rho \frac{\partial r}{\partial \chi}\\
    \frac{\partial \dot{m}}{\partial \chi}=8\pi r\dot{r}(\chi) \rho\frac{\partial r}{\partial \chi}+4\pi r^2 \dot{\rho}\frac{\partial r}{\partial \chi}+4\pi r^2 \rho\frac{\partial \dot{r}}{\partial \chi}
  \end{gather}
\end{subequations}
From \eqref{eq:rho:n} we also have
\begin{equation}
  \dot{\rho}=n\dot{n}
\end{equation}
Combining the above we obtain the following o.d.e.~for the variation of $\dot{m}$:
\begin{align}\label{eq:mdot}
      \frac{\partial \dot{m}}{\partial \chi}+\frac{1}{4\pi r^2 \frac{\partial r}{\partial \chi}}\frac{\dot{m}}{r}=&\,8\pi r\dot{r} \rho\frac{\partial r}{\partial \chi}+\frac{-2+5\frac{m}{r}}{4\pi r^3 \frac{\partial r}{\partial \chi}}\dot{r}+4\pi r^2 \bigl(\rho-n^2\bigr)\frac{\partial \dot{r}}{\partial \chi}\\
      \tag{by (\ref{eq:rho:n}}=&\,8\pi r\dot{r} \rho\frac{\partial r}{\partial \chi}+\frac{-2+5\frac{m}{r}}{4\pi r^3 \frac{\partial r}{\partial \chi}}\dot{r}-4\pi r^2 \bigl(\rho-1\bigr)\frac{\partial \dot{r}}{\partial \chi}
\end{align}
which we can solve for $\dot{m}$ using integrating factors and integrating by parts the $\frac{\ud\dot{r}}{\ud\chi}$ term:
\begin{align}\label{eq:m:dot:int}
  &\dot{m}(\chi) \exp\big\{\int_0^\chi\frac{\ud \chi}{4\pi r^3 \frac{\ud r}{\ud\chi}}\big\}\\
\notag  =&\int_0^\chi\bigg[8\pi r\dot{r} \rho\frac{\ud r}{\ud \chi}+\frac{-2+5\frac{m}{r}}{4\pi r^3 \frac{\partial r}{\partial \chi}}\dot{r}-4\pi r^2 \bigl(\rho-1\bigr)\frac{\ud \dot{r}}{\ud \chi}\bigg] \exp\big\{\int_0^{\chi'}\frac{\ud \chi}{4\pi r^3 \frac{\ud r}{\ud \chi}}\big\}\ud\chi'\\
 \notag =&\int_0^\chi\bigg[8\pi r (2\rho-1)\frac{\ud r}{\ud \chi}+\frac{-2+5\frac{m}{r}}{4\pi r^3 \frac{\ud r}{\ud \chi}}+4\pi r^2 \frac{\ud \rho}{\ud\chi}+\frac{\rho-1}{r\frac{\ud r}{\ud \chi}}\bigg] \dot{r}\exp\big\{\int_0^{\chi'}\frac{\ud \chi}{4\pi r^3 \frac{\ud r}{\ud \chi}}\big\}\ud\chi'\\
\notag&-4\pi r^2(\rho-1)\dot{r}\exp\big\{\int_0^\chi\frac{\ud \chi}{4\pi r^3 \frac{\ud r}{\ud \chi}}\big\}\bigg|_{\chi}
\end{align}
A remark is in order here for the integrability of the exponent of the integrating factor. Since $n=2\rho-2\ge1$, $r>2M$, from \eqref{eq:n:chi} it follows that
$\inf r^4(\frac{\ud r}{\ud \chi})^2>0$. Hence, the exponentiated integral is finite and in fact of order $r^2$, so the previous derivation is legitimate.

Evaluating both sides of \eqref{eq:m:dot:int} at $\chi=B$, the boundary term vanishes by virtue of the boundary condition:
\begin{equation}
  \rho-1=p,\qquad p=0\quad\text{: on the boundary.}
\end{equation}
Since $m(B)=M$ is the total mass, we see that the given distribution is a critical point of the mass functional, $\dot{M}=0$ for all $\dot{r}$, if and only if
\begin{align} \label{eq:dp:critical}
0=&\,8\pi r (2\rho-1)\frac{\ud r}{\ud \chi}+\frac{-2+5\frac{m}{r}}{4\pi r^3 \frac{\ud r}{\ud \chi}}+4\pi r^2 \frac{\ud \rho}{\ud\chi}+\frac{\rho-1}{r\frac{\ud r}{\ud \chi}}\\
\tag{using \eqref{eq:n:chi}}    \frac{\ud \rho}{\ud r}=&-\frac{2}{r} (2\rho-1)-n^2\frac{-2+5\frac{m}{r}}{r(1-\frac{2m}{r}) }-n^24\pi r^2\frac{\rho-1}{r(1-\frac{2m}{r})}\\
\tag{$n^2=2\rho-1$}  =&-\frac{2\rho-1}{r-2m}\bigg(2\frac{r-2m}{r}-2+5\frac{m}{r}+4\pi r^2(\rho-1)\bigg)\\
\notag=&-\frac{2\rho-1}{r-2m}\bigg(\frac{m}{r}+4\pi r^2(\rho-1)\bigg)
\end{align}
as asserted. Note that
\begin{equation}
  \dot{M}=\dot{m}(B)=\frac{\partial m}{\partial \lambda}(\chi=B)\rvert_{\lambda=0}
\end{equation}
is only true in the chosen coordinates under the particle constraint, because then  $\chi=B=N$ is indeed the  boundary of the family of solutions.
\end{proof}

Returning to \eqref{eq:m:dot:int} and using \eqref{eq:dp:critical} we see that for any critical point of the mass functional the following variational formula for $\dot{m}$ holds pointwise.
\begin{cor}\label{cor:dot:m}
  Let $\{(r,\rho)\}$ be a family of solutions to the hard phase equations with constant particle number $N$, 
  and assume that $(r=r_0,\rho=\rho_0)$ is a \emph{critical point} of the mass functional. Then
\begin{align}\label{dotmformula}
\dot{m}(\chi)=-4\pi r^2_0(\rho_0-1)\dot{r}(\chi), \qquad \chi\in [0,N]
\end{align}
where  $\chi$ is defined as in \eqref{eq:chi:number}.
\end{cor}

\subsection{Second variation}
\label{sec:variation:second}

We will show that hard  static stars with sufficiently small radius $R$ lie in a local minimum of the mass energy functional.
\begin{prop}\label{prop:critical}
The static  critical points of the mass functional \eqref{def:M} are local minima, i.e., $\ddot{M}[r]\ge0$, for $R=r(B)>0$ sufficiently small. Moreover $\ddot{M}[r]=0$ if and only if $\dot{r}=\partial_\phi\dot{r}=0$.
\end{prop} 
\begin{proof}
Unlike the derivations for the first variation in the previous subsection, we have take into account the contribution of $\frac{\partial r}{\partial\phi}$ in $\ddot{M}$. 
Since $\frac{\partial r}{\partial\phi}$ enters as a square in \eqref{eq:n:chi} and $\frac{\partial r}{\partial\phi}$ is zero for the static background, from the terms containing derivatives in $\partial_\phi$ only a kinetic term containing $(\frac{\partial \dot{r}}{\partial\phi})^2$ will survive in the second variation.  

We compute
\begin{align}
\frac{\partial \ddot{m}}{\partial \chi}=&\;8\pi(\dot{r})^2\rho\frac{\partial r}{\partial \chi}+8\pi r\ddot{r}\rho\frac{\partial r}{\partial \chi}+\frac{2}{r}\sqrt{1-\frac{2m}{r}}\dot{r}\dot{n}+16\pi r\dot{r}\rho\frac{\partial \dot{r}}{\partial \chi}\\
\notag&+\frac{\sqrt{1-\frac{2m}{r}}}{\frac{\partial r}{\partial\chi}}\frac{\partial\dot{r}}{\partial\chi}\dot{n}+4\pi r^2\rho\frac{\partial\ddot{r}}{\partial\chi}
+\frac{4\pi r(\rho-1)+\frac{m}{r^2}}{\sqrt{1-\frac{2m}{r}}}\dot{r}\dot{n}+\sqrt{1-\frac{2m}{r}}\ddot{n}
\end{align}
and
\begin{align}
\ddot{n}=&\;\frac{1}{\sqrt{1-\frac{2m}{r}}}\frac{-\frac{\ddot{m}}{r}+\frac{m}{r^2}\ddot{r}}{4\pi r^2 \frac{\partial r}{\partial \chi}}-2\frac{\sqrt{1-\frac{2m}{r}}}{4\pi r^3 \frac{\partial r}{\partial \chi}}\ddot{r}-\frac{\sqrt{1-\frac{2m}{r}}}{4\pi r^2\bigl( \frac{\partial r}{\partial \chi} \bigr)^2}\frac{\partial \ddot{r}}{\partial \chi}\\
\notag&-\frac{(4\pi r(\rho-1)+\frac{m}{r^2})^2}{(1-\frac{2m}{r})^\frac{3}{2}4\pi r^2 \frac{\partial r}{\partial \chi}}(\dot{r})^2-6\frac{4\pi r(\rho-1)+\frac{m}{r^2}}{\sqrt{1-\frac{2m}{r}}4\pi r^3 \frac{\partial r}{\partial \chi}}(\dot{r})^2
+4\frac{\sqrt{1-\frac{2m}{r}}}{4\pi r^3(\frac{\partial r}{\partial\chi})^2}\dot{r}\frac{\partial\dot{r}}{\partial\chi}\displaybreak[0]\\
\notag&-2\frac{4\pi r(\rho-1)+\frac{m}{r^2}}{\sqrt{1-\frac{2m}{r}}4\pi r^2(\frac{\partial r}{\partial\chi})^2}\dot{r}\frac{\partial\dot{r}}{\partial\chi}
+6\frac{\sqrt{1-\frac{2m}{r}}}{4\pi r^4\frac{\partial r}{\partial\chi}}(\dot{r})^2
+2\frac{\sqrt{1-\frac{2m}{r}}}{4\pi r^2(\frac{\partial r}{\partial\chi})^3}(\frac{\partial \dot{r}}{\partial\chi})^2\\
\notag&+\frac{1}{\sqrt{1-\frac{2m}{r}}}\frac{2\rho-1}{4\pi r^2\frac{\partial r}{\partial\chi}}(\frac{\partial \dot{r}}{\partial\phi})^2
\end{align}
By \eqref{eq:ndot} we also have
\begin{align}
\dot{r}\dot{n}=\frac{4\pi r(\rho-1)+\frac{m}{r^2}}{\sqrt{1-\frac{2m}{r}}4\pi r^2 \frac{\partial r}{\partial \chi}}(\dot{r})^2-2\frac{\sqrt{1-\frac{2m}{r}}}{4\pi r^3 \frac{\partial r}{\partial \chi}}(\dot{r})^2-\frac{\sqrt{1-\frac{2m}{r}}}{4\pi r^2\bigl( \frac{\partial r}{\partial \chi} \bigr)^2}\dot{r}\frac{\partial \dot{r}}{\partial \chi}
\end{align}
Hence, the equation for $\frac{\partial \ddot{m}}{\partial \chi}$ reads:
\begin{align}\label{eq:mddot}
\frac{\partial \ddot{m}}{\partial \chi}=&\;8\pi r\ddot{r}\rho\frac{\partial r}{\partial \chi}
+\frac{-\frac{\ddot{m}}{r}+\frac{m}{r^2}\ddot{r}}{4\pi r^2 \frac{\partial r}{\partial \chi}}-2\frac{1-\frac{2m}{r}}{4\pi r^3 \frac{\partial r}{\partial \chi}}\ddot{r}-4\pi r^2(\rho-1)\frac{\partial\ddot{r}}{\partial\chi}\\
\notag&+8\pi\rho\frac{\partial r}{\partial \chi}(\dot{r})^2+8\pi r\rho\frac{\partial}{\partial\chi}(\dot{r})^2-6\frac{4\pi r(\rho-1)+\frac{m}{r^2}}{4\pi r^3 \frac{\partial r}{\partial \chi}}(\dot{r})^2\\
\notag&+2\frac{1-\frac{2m}{r}}{4\pi r^4\frac{\partial r}{\partial\chi}}(\dot{r})^2-2\frac{4\pi r(\rho-1)+\frac{m}{r^2}}{4\pi r^2(\frac{\partial r}{\partial\chi})^2}\dot{r}\frac{\partial\dot{r}}{\partial\chi}
+\frac{1-\frac{2m}{r}}{4\pi r^2(\frac{\partial r}{\partial\chi})^3}(\frac{\partial \dot{r}}{\partial\chi})^2\\
\notag&+\frac{2\rho-1}{4\pi r^2\frac{\partial r}{\partial\chi}}(\frac{\partial \dot{r}}{\partial\phi})^2
\end{align}
Note that the double dotted terms in \eqref{eq:mddot} are exactly analogous to the ones in the first variation equation \eqref{eq:mdot} of $\frac{\partial \dot{m}}{\partial\chi}$. Using integrating factors as in \eqref{eq:m:dot:int} and integrating by parts, the resulting expression containing double dotted terms vanishes identically due to the fact that $\rho$ satisfies \eqref{eq:drhodr}. Thus, we arrive at the equation:
\begin{align}\label{ddotMformula}
&\ddot{M}\exp\bigg[\int_0^B\frac{d\chi}{4\pi r^3\frac{\partial r}{\partial\chi}}\bigg]\\
\notag=&\int_0^{B}\bigg[8\pi\rho\frac{\partial r}{\partial \chi}(\dot{r})^2+16\pi r\rho\dot{r}\frac{\partial\dot{r}}{\partial\chi}-6\frac{4\pi r(\rho-1)+\frac{m}{r^2}}{4\pi r^3 \frac{\partial r}{\partial \chi}}(\dot{r})^2
+2\frac{1-\frac{2m}{r}}{4\pi r^4\frac{\partial r}{\partial\chi}}(\dot{r})^2\\
\notag&-2\frac{4\pi r(\rho-1)+\frac{m}{r^2}}{4\pi r^2(\frac{\partial r}{\partial\chi})^2}\dot{r}\frac{\partial\dot{r}}{\partial\chi}
+\frac{1-\frac{2m}{r}}{4\pi r^2(\frac{\partial r}{\partial\chi})^3}(\frac{\partial \dot{r}}{\partial\chi})^2+\frac{2\rho-1}{4\pi r^2\frac{\partial r}{\partial\chi}}(\frac{\partial \dot{r}}{\partial\phi})^2\bigg]\cdot\\
\notag&\cdot\exp\bigg[\int^{\chi'}_0\frac{d\chi'}{4\pi r^3\frac{\partial r}{\partial\chi}}\bigg]d\chi
\end{align}
Now we make use of the \emph{smallness condition} for the stars by appealing to the behaviour at the center as given in  Proposition \ref{prop:exist:stars}: For $0\leq r\leq R\ll1$ we have the asymptotics $m=O(r^3)$, $1-\frac{2m}{r}\sim 1,4\pi
r(\rho-1)+\frac{m}{r^2}\sim r$. Observe that as $r\rightarrow0$ the dominant zeroth order dot term in the last integral is 
\begin{align*}
2\frac{1-\frac{2m}{r}}{4\pi r^4\frac{\partial r}{\partial\chi}}(\dot{r})^2\sim \frac{\dot{r}^2}{r^2}
\end{align*}
and having a positive sign. We use the latter to absorb the negative $(\dot{r})^2$ terms and we handle the mixed dotted term $\dot{r}\frac{\partial\dot{r}}{\partial\chi}$ by applying Cauchy's inequality
\begin{align}\label{mixeddotCauchy}
-2\frac{4\pi r(\rho-1)+\frac{m}{r^2}}{4\pi r^2\frac{\partial r}{\partial\chi}}\dot{r}\frac{\partial\dot{r}}{\partial\chi}\geq -\frac{\varepsilon}{r^2(\frac{\partial r}{\partial\chi})^2}(\frac{\partial\dot{r}}{\partial\chi})^2- \frac{C}{\varepsilon}(\dot{r})^2
\end{align}
Notice that the term containing $(\frac{\partial \dot{r}}{\partial\phi})^2$ is manifestly positive, since $\rho\ge1$. This term represents the ``kinetic energy'' of deviations from equilibrium to second order. We conclude that $\ddot{M}[r]\ge0$, for static hard stars with small radius. Finally the case $\ddot{M}[r]=0$ holds if and only if $\dot{r}=0=\partial_\phi\dot{r}$.
\end{proof}

\subsection{Quantities controlled by the second variation}

{\it Notation}: We denote all the variables of the static solution to the hard phase by a $\rho_0,r_0,m_0$ etc.  Relations of the form $a\sim b$ below mean that there exist constants $c,C>0$ such that $cb\leq a\leq Cb$. 
\begin{prop}\label{ddotMenergy}
The second variation of the mass functional $M[r]$ of small hard stars is equivalent to the following energy:
\begin{align}\label{2ndvaren}
\ddot{M}[r_0]\sim\int_0^{B}\frac{\dot{r}^2}{r_0^2}+r_0^4(\frac{\partial\dot{r}}{\partial\chi})^2+(\frac{\partial\dot{r}}{\partial\phi})^2d\chi
\end{align}
\end{prop}
The proof makes reference to the leading order behaviour of the background solution at the center, which we summarize in the following Lemma:

\begin{lemma}\label{lemma:stars:asymptotics}
For  hard stars of \emph{small} radius $R_0$,
\begin{subequations}
\begin{gather}\label{leadasym}
 n_0,\rho_0=1+O(R_0^2-r_0^2),\qquad m_0=\frac{4\pi}{3}r_0^3[1+O(R_0^2-r_0^2)]\\
\label{leadasymb}\frac{\partial r_0}{\partial\chi}=\frac{1+O(R_0^2)}{4\pi r_0^2}+O(1),\qquad \chi \sim r_0^\frac{1}{3},\\ 
\label{leadasymc}e^{-\omega_0}=4\pi r_0^2[1+O(R_0^2-r_0^2)],\;\;\;
\partial_\chi \rho_0=-\frac{\frac{1}{3}+O(R_0^2-r_0^2)}{r_0}+O(r_0),
\end{gather}
\end{subequations}
where the remainders are analytic functions in $r_0$ and their derivatives satisfy analogous bounds.
\end{lemma}

\begin{proof}
  The first estimate follows from \eqref{eq:rho:approx} and \eqref{eq:rho:n}.
  Plugging it subsequently into \eqref{eq:m:chi} and \eqref{eq:n:chi} we deduce the behaviours for $m_0$ and $\partial_\chi r_0$ respectively. The rate of $\partial_\chi r_0$ also yields the relation between $\chi$ and $r_0$. The function $e^{-\omega_0}$ is given by \eqref{def:n} in terms of $n_0,r_0$ and the asymptotic behaviour of $\partial_\chi\rho_0$ is computed from the TOV equation \eqref{eq:rho:ode} using the estimates for $m_0,\rho_0,\partial_\chi r_0$.
\end{proof}
\begin{proof}[Proof of the Proposition \ref{ddotMenergy}]
Going back to \eqref{ddotMformula} and replacing the background variables with their leading orders according to the previous lemma, we see from \eqref{mixeddotCauchy} that $\ddot{M}[r_0]$ controls the RHS of \eqref{2ndvaren}. The fact that $\ddot{M}[r_0]$ is also bounded by the aforementioned energy is evident.
\end{proof}
Note that the energy space \eqref{2ndvaren} is quite weak in that it does not even control the $\mathrm{L}^\infty$ norm of $\dot{r}$, as $r_0\rightarrow0$. This is not surprising as energies coming from conserved quantities in the context of the Einstein equations in spherical symmetry are generally fairly weak. An example of such an energy can be written down for the \textsl{AdS scalar field model}, where the $\dot{H}^1$ energy of the scalar field is conserved in the evolution. However, the latter does not provide adequate control of the solution. Indeed it is compatible with black hole formation,\footnote{In this case the mass aspect function $\frac{m}{r}$ becoming larger than $\frac{1}{2}$.} see \cite{Bizon:11} and the recent \cite{Moschidis:18}.

On the other hand, in the present context of small hard stars, we may use \eqref{2ndvaren} together with the first variation \eqref{cor:dot:m} of $m$ to obtain a satisfactory control of $\frac{m}{r}$ to first order. This crucially relies on the fact that we are linearising around a static hard star and it highlights the particularly ``stable structure'' that these equilibria enjoy. \footnote{Indeed, it was shown by Christodoulou \cite{Ch:93} that the quantity $\frac{m}{r}$ serves as a continuation criterion for the Einstein-massless scalar field equations in spherical symmetry, and  it can be expected that owing to the similarity of the equations --- c.f.~\eqref{Tmunu} --- the analogous criterion will be instrumental for the study of the hard phase model.}
\begin{prop}\label{moverrprop}
  Let $m_\lambda(\chi)$, $r_\lambda(\chi)$ be respectively the mass, and radius of a variation through a static hard star $(\lambda=0)$, defined on $\chi\in[0,B]$. Then the following estimate holds true:
\begin{align}\label{moverrest}
\big|(\frac{m}{r^{1+\epsilon}})\dot{}\,\big|\leq Cr^{\frac{1}{2}-\epsilon}_0\ddot{M}[r_0],&&0\leq \epsilon\leq\frac{1}{2},
\end{align}
in the whole domain $[0,B]$.
\end{prop}
\begin{proof}
Employing (\ref{dotmformula}) we compute
\begin{align}\label{varm/r}
  (\frac{m}{r^{1+\epsilon}})\dot{}=-4\pi r_0^{1-\epsilon}(\rho_0-1)\dot{r}-(1+\epsilon)\frac{m_0}{r_0^{2+\epsilon}}\dot{r}\sim r_0^{1-\epsilon}\dot{r}
\end{align}
which by (\ref{2ndvaren}) implies that
\begin{align}\label{L2varm}
\int^{B}_0\big[r^{\epsilon-2}_0(\frac{m}{r^{1+\epsilon}})\dot{}\big]^2d\chi\leq C\ddot{M}[\rho_0,r_0]\,.
\end{align}
Differentiating in $\chi$ and commuting with the dots we also have
\begin{align}\label{dchivarm}
[\frac{\partial}{\partial\chi}(\frac{m}{r^{1+\epsilon}})]\,\dot{}=&-4\pi (1-\epsilon)r_0^{\epsilon}\frac{\partial r_0}{\partial\chi}(\rho_0-1)\dot{r}-4\pi r_0^{1-\epsilon}\frac{\partial\rho_0}{\partial\chi}\dot{r}\\
\notag&-4\pi r_0^{1-\epsilon}(\rho_0-1)\frac{\partial\dot{r}}{\partial\chi}
-(1+\epsilon)\frac{\partial_\chi m_0}{r_0^{2+\epsilon}}\dot{r}\\
\notag&-(1+\epsilon)(2+\epsilon)\frac{m_0}{r_0^{3+\epsilon}}\frac{\partial r_0}{\partial\chi}\dot{r}-(1+\epsilon)\frac{m_0}{r_0^{2+\epsilon}}\frac{\partial\dot{r}}{\partial\chi}
\end{align}
Using the behaviour at the center $m_0\sim r_0^3,\partial_\chi m_0\sim1$ and $\frac{\partial r_0}{\partial\chi}\sim r_0^{-2}$ we conclude that 
\begin{align}\label{L2vardmdchi}
\int^{B}_0\big(r_0^{1+\epsilon}[\frac{\partial}{\partial\chi}(\frac{m}{r^{1+\epsilon}})]\,\dot{}\big)^2d\chi\lesssim \ddot{M}[\rho_0,r_0]
\end{align}
Finally, we combine \eqref{L2varm}, \eqref{L2vardmdchi} and apply the fundamental theorem of calculus to derive the pointwise bound:
\begin{align}\label{varmLinfty}
r_0^{\epsilon-\frac{1}{2}}\big[(\frac{m}{r^{1+\epsilon}})\dot{}\big]^2=&\,r_0^{\epsilon-\frac{1}{2}}\big[(\frac{m}{r^{1+\epsilon}})\dot{}\big]^2\bigg|_{\chi=0}\\
\notag&+\int^{\chi}_0(\epsilon-\frac{1}{2})r_0^{\epsilon-\frac{3}{2}}\frac{\partial r_0}{\partial\chi}\big[(\frac{m}{r^{1+\epsilon}})\dot{}\big]^2+2r_0^{\epsilon-\frac{1}{2}}(\frac{m}{r^{1+\epsilon}})\dot{}\frac{\partial}{\partial\chi}[(\frac{m}{r^{1+\epsilon}})\dot{}]\,d\chi\\
\tag{by \emph{Cauchy-Schwarz}}\leq&\,C\ddot{M}[\rho_0,r_0]
\end{align}
as long as $\epsilon-\frac{3}{2}-2\ge 2(1+\epsilon)-6$ or \underline{$\epsilon\leq \frac{1}{2}$}. Note that for $\epsilon\leq\frac{1}{2}$, it is immediate from (\ref{L2varm}) that there exists a sequence $\chi_\nu\rightarrow0$ such that $r_0^{\epsilon-\frac{1}{2}}\big[(\frac{m}{r^{1+\epsilon}})\dot{}\big]^2\big|_{\chi_\nu}\rightarrow0$, hence, the preceding calculation can be made rigorous by passing to the limit.
\end{proof}

\section{Linear theory}
\label{sec:linear}

The comoving coordinates  \eqref{eq:metric:com} that have already been introduced in Section~\ref{sec:variation} are particularly useful for the study of the dynamics of hard stars. In coordinates $(\phi,\chi)$, where $\phi$ is the time-function introduced in \eqref{eq:V:phi}, and $\chi$ parametrizes the flow lines of the fluid in $\mathcal{Q}$, the Hessian equations take the form (see (3.33) in \cite{Ch:95:I}):
\begin{subequations}\label{eqs:hessian:comoving}
	\begin{align}
\label{Hessphiphi}		e^{-2\psi}\Bigl(\frac{\partial^2 r}{\partial \phi^2}-\frac{\partial \psi}{\partial \phi}\frac{\partial r}{\partial \phi}\Bigr) -e^{-2\omega}\frac{\partial \psi}{\partial \chi}\frac{\partial r}{\partial \chi}=&-\frac{m}{r^2}-4\pi rp\\
\label{Hessphichi}		\frac{\partial^2 r}{\partial \phi\partial \chi}-\frac{\partial \omega}{\partial \phi}\frac{\partial r}{\partial \chi}-\frac{\partial \psi}{\partial \chi}\frac{\partial r}{\partial \phi}=&0\\
\label{Hesschichi}		e^{-2\omega}\Bigl(\frac{\partial^2 r}{\partial \chi^2}-\frac{\partial \omega}{\partial \chi}\frac{\partial r}{\partial \chi}\Bigr)-e^{-2\psi}\frac{\partial \omega}{\partial \phi}\frac{\partial r}{\partial \phi}=&\frac{m}{r^2}-4\pi r\rho
	\end{align}
\end{subequations}
Recall also the conservation law \eqref{eq:wave:com}, and the relations \eqref{psi:rho},\eqref{eq:rho:n}. In comoving coordinates the Hawking mass $m$ is given by \eqref{eq:mass:comoving} and satisfies the differential equations \eqref{eq:m:phi},\eqref{eq:m:chi}.

\begin{rem}
It is not obvious that this system is locally well-posed, with the free boundary condition $p=0$ on $\chi=\chi_B$.
However, the early work of \textsc{Kind-Ehlers} \cite{Kind:93} can be adapted to derive a \emph{symmetric first order hyperbolic system} for the unknowns
\begin{align}
X=-N\psi, &&Y=\frac{\partial_\chi Tr}{\partial_\chi r}+2\frac{Tr}{r},
\end{align}
coupled to ODE equations in time for the rest of the variables, see \cite[\S3.1]{Kind:93}.
Here we are interested only in the linearisation of the system, and thus in this paper we do not pursue the well-posedness of the non-linear system further in this formulation.\footnote{Recall that local existence for the free-boundary problem of the two phase model was shown in \emph{null coordiantes} in \cite{Ch:95:I}, \cite{Ch:96:II}.}
\end{rem}

In what follows we will linearise the Einstein-Euler equations in comoving coordinates, in the form \eqref{eqs:hessian:comoving}. Alternatively, one could linearise the system in double null coordinates, given in the form \eqref{eq:hessian}, which at first sight seems more natural because it will involve the study of the linear wave equation \eqref{eq:wave} for $\phi$ on a fixed background. In comoving coordinates $\phi$ is eliminated as an unknown, however these coordinates have the decisive advantage of allowing us to identify the boundaries of a family of solutions.

\subsection{Linearized equations}
\label{sec:linear:equations}

We linearise the spherically symmetric equations \eqref{eqs:hessian:comoving} around a static hard star, while keeping the total number of particles $N$ fixed. We denote all variables of the background with a $0$-subscript, i.e., $r_0,\phi_0$ etc. In particular, we set:
\begin{align}\label{pertvar}
	\begin{split} 
		\psi=\psi_0+\psi_1,\qquad r=r_0+r_1,\qquad\rho=\,\rho_0+\rho_1\\
		m=m_0+m_1,\qquad \omega=\omega_0+\omega_1,\qquad n=n_0+n_1.
	\end{split} 
\end{align}
The above definitions make sense after identifying the domain of the two sets of variables $\{r,\rho,\ldots\},\{r_0,\rho_0,\ldots\}$ via the diffeomorphism:
\begin{align}\label{diffeo}
	(\phi,\chi)\leftrightarrow(\phi_0,\chi_0),&&\chi=\int^\chi_0 n\ud \mu_{\phi=\mathrm{const}},\;\chi_0=\int_0^{\chi_0} n_0\ud\mu_{\phi_0=\mathrm{const}},
\end{align} 
where $r(0)=r_0(0)=0$. 

Recall that by keeping the total number of particles fixed, $N=N_0$, the spatial variables $\chi,\chi_0$ are \emph{a priori}
defined on the same interval $[0,B]$, where $\chi=\chi_0=B$ corresponds to the identified boundary, i.e., 
\begin{align}\label{rho1bd}
	r(B)=R,\;r_0(B)=R_0,\qquad\rho(B)=\rho_0(B)=1\;\Longrightarrow\; \rho_1(B)=0.
\end{align}
Also, observe that the conservation law \eqref{eq:wave:com} is equivalent to 
\begin{align}\label{conslaw2}
	\partial_\phi(n \ud\mu_{\phi=\mathrm{const.}})=0,
\end{align}
which implies that the formulas for $\chi,\chi_0$ are valid for all $\phi\equiv\phi_0$. We will use this observation to show that the linearised Einstein-Euler equations around a hard star in spherical symmetry reduce to one master equation for the area radius function $r$, from which  all the other linearised variables may be computed. 
\begin{prop}\label{prop:mastereq}
  Given a fixed hard star solution to the Einstein-Euler system and spherically symmetric perturbations of the form \eqref{pertvar}, which preserve the total number of particles $N$,  the linearised equations for $r_1$, $\rho_1$, $\psi_1$, $n_1$, $m_1$, $\omega_1$ reduce to the following form:
  \begin{enumerate}[label=(\alph*)]
    \item The linearised variable $r_1$ satisfies the decoupled, homogeneous wave equation
	\begin{align}\label{linHessreq}
		&(2\rho_0-1)\partial^2_\phi r_1-[\frac{m_0}{r^2_0}+4\pi r_0(\rho_0-1)]\frac{1}{\partial_\chi\psi_0}\partial_\chi(\frac{\partial_{\chi}r_1}{\partial_\chi r_0})\\
		\notag=&\,\bigg[2\frac{\partial_\chi\psi_0}{\partial_\chi r_0}[\frac{m_0}{r^2_0}+4\pi r_0(\rho_0-1)]
		+2\frac{m_0}{r^3_0}+4\pi r_0(2\rho_0-1)(\frac{2}{r_0}-\frac{\partial_\chi\psi_0}{\partial_\chi r_0})\bigg]r_1 \\
		\notag&+[\frac{m_0}{r^2_0}+4\pi r_0(\rho_0-1)]\frac{1}{\partial_\chi \psi_0}\partial_\chi\big[(\frac{2}{r_0}-\frac{\partial_\chi\psi_0}{\partial_\chi r_0}) r_1\big]+ (4\pi r_0\rho_0-\frac{m_0}{r_0^2})\frac{\partial_\chi r_1}{\partial_\chi r_0}
	\end{align}
	with mixed boundary conditions 
	\begin{align}\label{linr1bdcond}
		r_1=0:\;\text{at $\chi=0$},&&\partial_\chi r_1=(\frac{\partial_\chi\psi_0}{\partial_\chi r_0}-\frac{2}{r_0})r_1:\;\text{at $\chi=B$.}
	\end{align}
	\item The linearised variables $\rho_1$, $\psi_1$, $n_1$, $m_1$, $\omega_1$ are explicit functions of $r_1$, and can be computed directly from the formulas:
	\begin{subequations}\label{restlinvar}
	\begin{align}
			\rho_1=&-\psi_1(2\rho_0-1)=n_1\sqrt{2\rho_0-1},\\
			\omega_1=&\,\frac{\partial_\chi r_1}{\partial_\chi r_0}-\frac{\partial_\chi\psi_0}{\partial_\chi r_0}r_1,\\
			\psi_1=&\,\Bigl(\frac{2}{r_0}-\frac{\partial_\chi\psi_0}{\partial_\chi r_0}\Bigr)r_1+\frac{\partial_\chi r_1}{\partial_\chi r_0}\\
			m_1=&-4\pi r^2_0(\rho_0-1)r_1\label{restlinvar:m}
	\end{align}
      \end{subequations}
    \end{enumerate}
\end{prop}
\begin{proof}
	Linearising \eqref{psi:rho} and \eqref{eq:rho:n} we obtain:
	\begin{align}\label{linrhopsi}
		\rho_1=-\psi_1(2\rho_0-1)=n_1\sqrt{2\rho_0-1}
	\end{align}
	Note that by \eqref{conslaw2} and our choice of spatial coordinate $\chi$ it follows that 
	\begin{align}\label{conslaw3}
		r^2e^{-\psi}e^\omega=\frac{1}{4\pi}
	\end{align}
	Therefore, the linearised equation reads:
	\begin{align}\label{linconslaw}
		\frac{2}{r_0} r_1- \psi_1+ \omega_1=0
	\end{align}
	Likewise, linearising the first equation in \eqref{Hessphiphi} we obtain
	\begin{multline}\label{linHessreq2}
		(2\rho_0-1)\partial^2_\phi r_1+(2\omega_1-\frac{\partial_\chi\psi_1}{\partial_\chi\psi_0}-\frac{\partial_\chi r_1}{\partial_\chi r_0})[\frac{m_0}{r_0^2}+4\pi r_0(\rho_0-1)]=\\
		=-\frac{m_1}{r_0^2}+2\frac{m_0}{r^3_0}r_1-4\pi r_1(\rho_0-1)-4\pi r_0\rho_1 
	\end{multline}
	On the other hand, linearising the mass equation \eqref{eq:mass:comoving} gives
	\begin{align}\label{linm}
		-\frac{m_1}{r_0}+\frac{m_0}{r^2_0}r_1=-(1-\frac{2m_0}{r_0})\omega_1+(1-\frac{2m_0}{r_0})\frac{\partial_\chi r_1}{\partial_\chi r_0},
	\end{align}
	which we may use to replace $\omega_1$ from \eqref{linconslaw},\eqref{linHessreq2} above in favour of $m_1$. 
	
	Recall that the variational properties of the static hard stars imply that the following formula is valid pointwise:
	\begin{align}\label{m1}
		m_1=-4\pi r_0^2(\rho_0-1)r_1\,,
	\end{align}
	which gives us in turn an expression for $\omega_1$ in terms of $r_1$ by \eqref{linm}: 
	\begin{align}\label{omega1}
		\omega_1=&\,\frac{\partial_\chi r_1}{\partial_\chi r_0}-\frac{\frac{m_0}{r^2_0}+4\pi r_0(\rho_0-1)}{1-\frac{2m_0}{r_0}}r_1
		=\frac{\partial_\chi r_1}{\partial_\chi r_0}-\frac{\partial_\chi\psi_0}{\partial_\chi r_0}r_1,
	\end{align}
	where in the last equality we used  \eqref{eq:rho:ode} expressed in terms of $\psi_0$, $\psi_0=-\frac{1}{2}\log(2\rho_0-1)$:
	\begin{subequations}
	\begin{align}\label{TOV}
			\partial_\chi\rho_0=&-\frac{2\rho_0-1}{1-\frac{2m_0}{r_0}}\Bigl[\frac{m_0}{r_0^2}+4\pi r_0(\rho_0-1)\Bigr]\partial_\chi r_0\\ \iff \partial_\chi\psi_0=&\,\frac{1}{1-\frac{2m_0}{r_0}}\Bigl[\frac{m_0}{r_0^2}+4\pi r_0(\rho_0-1)\Bigr]\partial_\chi r_0
	\end{align}
      \end{subequations}
      Note that $\partial_\chi\rho_0<0$, $\partial_\chi\psi_0>0$, as long as $r_0>2m_0$, $\rho_0\ge1$.
	
	Thus, substituting \eqref{m1},\eqref{omega1} into \eqref{linconslaw},\eqref{linHessreq2} we obtain:

	\begin{align}\label{psi1}
		\psi_1=\Bigl(\frac{2}{r_0}-\frac{\partial_\chi\psi_0}{\partial_\chi r_0}\Bigr)) r_1+ \frac{\partial_\chi r_1}{\partial_\chi r_0}
	\end{align}
	and
	\begin{align}\label{linHessreq3}
		&(2\rho_0-1)\partial^2_\phi r_1+(\frac{\partial_\chi r_1}{\partial_\chi r_0}-2\frac{\partial_\chi\psi_0}{\partial_\chi r_0}r_1-\frac{\partial_\chi\psi_1}{\partial_\chi \psi_0})[\frac{m_0}{r_0^2}+4\pi r_0(\rho_0-1)]\\
		\notag=&-\frac{m_1}{r_0^2}+2\frac{m_0}{r^3_0}r_1-4\pi r_1(\rho_0-1)-4\pi r_0\rho_1 \\
		\tag{by \eqref{linrhopsi},\eqref{m1}}=&\,[4\pi(\rho_0-1)+2\frac{m_0}{r^3_0}-4\pi (\rho_0-1)]r_1+4\pi r_0(2\rho_0-1)\psi_1\\
		\tag{by \eqref{psi1}}=&\,\big[2\frac{m_0}{r^3_0}+4\pi r_0(2\rho_0-1)(\frac{2}{r_0}-\frac{\partial_\chi\psi_0}{\partial_\chi r_0})\big] r_1+ 4\pi r_0(2\rho_0-1)\frac{\partial_\chi r_1}{\partial_\chi r_0}
	\end{align}
	Using \eqref{psi1} to also eliminate $\psi_1$ from the left hand side of \eqref{linHessreq3} we arrive at the homogeneous wave equation \eqref{linHessreq} for $r_1$.
\end{proof}
\subsection{Linear stability of hard stars}
\label{sec:linear:stability}

In this subsection we exploit the hierarchical form of the linearised equations exhibited by Proposition \ref{prop:mastereq} to derive energy estimates for the linearised variables. More precisely, 
we observe that in the small radius regime, the wave equation \eqref{linHessreq} for $r_1$ possesses a favourable leading order potential term, which along with a careful estimate of the coefficient of the first order term in the RHS of \eqref{linHessreq} enables us to obtain uniform global-in-time energy estimates for $r_1$. Thus, appealing also to \eqref{restlinvar}, we show that the linearised system has a uniformly bounded $H^1$-type energy.
\begin{prop}\label{prop:linenest}
	Let $r_1$ be a solution to \eqref{linHessreq} satisfying the boundary conditions \eqref{bdcondr1}. Then the following energy estimates hold:
	\begin{align}\label{linenest}
		\mathcal{E}(\phi):=\int^B_0\big[\frac{r_1^2}{r_0^2}+(\partial_\phi r_1)^2+r_0^4(\partial_\chi r_1)^2\big]d\chi \leq C\mathcal{E}(0)
	\end{align}
	and
	\begin{align}\label{linenest2}
		\int^B_0\rho_1^2+n_1^2+\psi_1^2+\omega_1^2d\chi\leq C\mathcal{E}(0),
	\end{align}
	for all $\phi\ge0$, where $C>0$ is a constant only depending on the background hard star solution.
\end{prop}
\begin{rem}
	The energy estimates \eqref{linenest}, \eqref{linenest2} are exactly at the level of the energy controlled by the second variation of $M$, cf.~Proposition~\ref{ddotMenergy}. In particular, we may use \eqref{linenest} and \eqref{restlinvar:m} to derive the same pointwise and $H^1$ bounds for $m_1$ as derived for $\dot{m}$ in Proposition \ref{moverrprop}.
\end{rem}
\begin{proof} 
	We rewrite \eqref{linHessreq3} by expressing the second term in the LHS as a pure $\partial_\chi$ derivative and using \eqref{TOV}:
	\begin{align}\label{linHessreq4}
		&(2\rho_0-1)\partial^2_\phi r_1-\partial_\chi\bigg[\frac{1-\frac{2m_0}{r_0}}{\partial_\chi r_0}\frac{\partial_{\chi}r_1}{\partial_\chi r_0}\bigg]\\
		\notag=&\,\bigg[2\frac{\partial_\chi\psi_0}{\partial_\chi r_0}[\frac{m_0}{r^2_0}+4\pi r_0(\rho_0-1)]
		+2\frac{m_0}{r^3_0}+4\pi r_0(2\rho_0-1)(\frac{2}{r_0}-\frac{\partial_\chi\psi_0}{\partial_\chi r_0})\bigg]r_1 \\
		\notag&+\frac{1-\frac{2m_0}{r_0}}{\partial_\chi r_0}\partial_\chi\big[(\frac{2}{r_0}-\frac{\partial_\chi\psi_0}{\partial_\chi r_0}) r_1\big]+ (4\pi r_0\rho_0-\frac{m_0}{r_0^2})\frac{\partial_\chi r_1}{\partial_\chi r_0}
		-\frac{\partial_{\chi}r_1}{\partial_\chi r_0}\partial_\chi\bigg(\frac{1-\frac{2m_0}{r_0}}{\partial_\chi r_0}\bigg)\\
		\notag=&\,\bigg[2\frac{\partial_\chi\psi_0}{\partial_\chi r_0}[\frac{m_0}{r^2_0}+4\pi r_0(\rho_0-1)]
		+2\frac{m_0}{r^3_0}+4\pi r_0(2\rho_0-1)(\frac{2}{r_0}-\frac{\partial_\chi\psi_0}{\partial_\chi r_0})\\
		\notag&+\frac{1-\frac{2m_0}{r_0}}{\partial_\chi r_0}\partial_\chi(\frac{2}{r_0}-\frac{\partial_\chi\psi_0}{\partial_\chi r_0})\bigg]r_1\\ 
		\notag&+\bigg[(1-\frac{2m_0}{r_0})(\frac{2}{r_0}-\frac{\partial_\chi\psi_0}{\partial_\chi r_0})+4\pi r_0\rho_0-\frac{m_0}{r_0^2}
		-\partial_\chi\bigg(\frac{1-\frac{2m_0}{r_0}}{\partial_\chi r_0}\bigg)\bigg]\frac{\partial_{\chi}r_1}{\partial_\chi r_0}
	\end{align}
	Next, we compute the coefficient of the last term by using \eqref{Hesschichi},\eqref{eq:m:chi}:
	\begin{align}\label{statid2}
		&-\partial_\chi\bigg(\frac{1-\frac{2m_0}{r_0}}{\partial_\chi r_0}\bigg)=\frac{1-\frac{2m_0}{r_0}}{(\partial_\chi r_0)^2}\partial_\chi^2r_0+\frac{\frac{2\partial_\chi m_0}{r_0}-\frac{2m_0}{r_0^2}\partial_\chi r_0}{\partial_\chi r_0}\\
		\notag=&\,\frac{1-\frac{2m_0}{r_0}}{(\partial_\chi r_0)^2}\Bigl(\partial_\chi\omega_0\partial_\chi r_0+e^{2\omega_0}\frac{m_0}{r_0^2}-e^{2\omega_0}4\pi r_0\rho_0\Bigr)+8\pi r_0\rho_0-\frac{2m_0}{r_0^2}\\
		\notag=&\,\frac{1-\frac{2m_0}{r_0}}{(\partial_\chi r_0)^2}\partial_\chi\omega_0\partial_\chi r_0+\frac{m_0}{r_0^2}-4\pi r_0\rho_0+8\pi r_0\rho_0-\frac{2m_0}{r_0^2}\\
		\tag{by \eqref{conslaw3}} =&\,\frac{1-\frac{2m_0}{r_0}}{(\partial_\chi r_0)^2}\Bigl(\partial_\chi\psi_0-\frac{2}{r_0}\partial_\chi r_0\Bigr)\partial_\chi r_0+4\pi r_0\rho_0-\frac{m_0}{r_0^2}\\
		\notag=&\,\bigl(1-\frac{2m_0}{r_0}\bigr)\bigl(\frac{\partial_\chi\psi_0}{\partial_\chi r_0}-\frac{2}{r_0}\bigr)+4\pi r_0\rho_0-\frac{m_0}{r_0^2}
	\end{align}
	Hence, the wave equation \eqref{linHessreq4} becomes
	\begin{align}\label{linHessreq5}
		&(2\rho_0-1)\partial^2_\phi r_1-\partial_\chi\bigg[\frac{1-\frac{2m_0}{r_0}}{\partial_\chi r_0}\frac{\partial_{\chi}r_1}{\partial_\chi r_0}\bigg]\\
		\notag=&\,\bigg[2\frac{\partial_\chi\psi_0}{\partial_\chi r_0}[\frac{m_0}{r^2_0}+4\pi r_0(\rho_0-1)]
		+2\frac{m_0}{r^3_0}+4\pi r_0(2\rho_0-1)\Bigl(\frac{2}{r_0}-\frac{\partial_\chi\psi_0}{\partial_\chi r_0}\Bigr)\\
		\notag&+\frac{1-\frac{2m_0}{r_0}}{\partial_\chi r_0}\partial_\chi\Bigl(\frac{2}{r_0}-\frac{\partial_\chi\psi_0}{\partial_\chi r_0}\Bigr)\bigg]r_1 
		+\bigl(8\pi r_0\rho_0-\frac{2m_0}{r_0^2}\bigr)\frac{\partial_{\chi}r_1}{\partial_\chi r_0}
	\end{align}

	We wish to eliminate the last term in the RHS by making a transformation $r_1\to f r_1$. The correct function $f$ is found via integrating factors. It is important however that $f$ is integrable in $[0,B]$, otherwise this procedure might introduce undesirable weights in the energy estimate below. 
        We have
	\begin{align}\label{linHessreq6}
		&e^{\int^\chi_0fd\chi'}(2\rho_0-1)\partial^2_\phi r_1-\partial_\chi\bigg[e^{\int^\chi_0fd\chi'}\frac{1-\frac{2m_0}{r_0}}{\partial_\chi r_0}\frac{\partial_{\chi}r_1}{\partial_\chi r_0}\bigg]\\
		\notag=&\,e^{\int^\chi_0fd\chi'}\bigg[2\frac{\partial_\chi\psi_0}{\partial_\chi r_0}\Bigl(\frac{m_0}{r^2_0}+4\pi r_0(\rho_0-1)\Bigr)
		+2\frac{m_0}{r^3_0}+4\pi r_0(2\rho_0-1)(\frac{2}{r_0}-\frac{\partial_\chi\psi_0}{\partial_\chi r_0})\\
		\notag&+\frac{1-\frac{2m_0}{r_0}}{\partial_\chi r_0}\partial_\chi(\frac{2}{r_0}-\frac{\partial_\chi\psi_0}{\partial_\chi r_0})\bigg]r_1,
	\end{align}
	for $f=(8\pi r_0\rho_0-\frac{2m_0}{r_0^2})(1-\frac{2m_0}{r_0})^{-1}\partial_\chi r_0$.
	
	From Lemma~\ref{lemma:stars:asymptotics} for the leading order behaviour of small stars we see that $f\sim r_0^{-1}\sim \chi^{-\frac{1}{3}}$, which makes $e^{\int^\chi_0fd\chi'}\sim1$. Moreover, we notice that the dominant coefficient of $r_1$ on the r.h.s~of \eqref{linHessreq6} is
	\begin{align}\label{domcoeffr1}
		\frac{1-\frac{2m_0}{r_0}}{\partial_\chi r_0}\partial_\chi(\frac{2}{r_0})\sim -\frac{2}{r_0^2}
\end{align}
Indeed, this follows by noticing that 
\begin{align}\label{lowcoeffr1}
\frac{\partial_\chi\psi_0}{\partial_\chi r_0}\overset{\eqref{TOV}}{=}\frac{1}{1-\frac{2m_0}{r_0}}\Bigl[\frac{m_0}{r_0^2}+4\pi r_0(\rho_0-1)\Bigr]\overset{\eqref{leadasym}}{\sim} r_0,\qquad |\partial_\chi(\frac{\partial_\chi\psi_0}{\partial_\chi r_0})|\lesssim r_0^{-2}
\end{align}
Thus, multiplying \eqref{linHessreq6} with $\partial_\phi r_1$, integrating in $[0,B]$ and integrating by parts we obtain the energy identity:
	\begin{align}\label{enineq}
		\notag&\frac{1}{2}\partial_\phi\int^B_0\bigg[e^{\int^{\chi'}_0f}(2\rho_0-1)(\partial_\phi r_1)^2+e^{\int^{\chi'}_0f}\frac{1-\frac{2m_0}{r_0}}{(\partial_\chi r_0)^2}(\partial_{\chi}r_1)^2\bigg]d\chi\\
		&-e^{\int^B_0f}\frac{1-\frac{2m_0}{r_0}}{(\partial_\chi r_0)^2}\partial_\chi r_1\partial_\phi r_1\bigg|_{\chi=B}
		=\int^B_0[-\frac{2}{r_0^2}+O(1)]r_1\partial_\phi r_1 d\chi
	\end{align}
	The relation \eqref{psi1} evaluated at $\chi=B$ gives
	\begin{align}\label{bdcondr1}
		\partial_\chi r_1=(\frac{\partial_\chi\psi_0}{\partial_\chi r_0}-\frac{2}{r_0})r_1,&&\text{at $\chi=B$}.
	\end{align}
	Moreover we have $\frac{\partial_\chi\psi_0}{\partial_\chi r_0}-\frac{2}{r_0}\sim -\frac{2}{r_0}$. Therefore, the boundary term in \eqref{enineq} satisfies
	\begin{align}\label{bdterm}
		-e^{\int^B_0f}\frac{1-\frac{2m_0}{r_0}}{(\partial_\chi r_0)^2}\partial_\chi r_1\partial_\phi r_1\bigg|_{\chi=B}\sim R_0^3 \,\partial_\phi(r_1^2),
	\end{align}
	provided $R_0$ is sufficiently small. Hence, all terms in \eqref{enineq} have a favourable sign, yielding the energy estimate:
	\begin{align}\label{enineq2}
		&\int^B_0\Bigl[\frac{r_1^2}{r_0^2}+(\partial_\phi r_1)^2+r_0^4(\partial_{\chi}r_1)^2\Bigr]\ud\chi + R_0^3r_1^2\big|_{\chi=B}\\
		\notag\leq &\,C\bigg[\int^B_0\frac{r_1^2}{r_0^2}+(\partial_\phi r_1)^2+r_0^4(\partial_{\chi}r_1)^2d\chi + R_0^3r_1^2\big|_{\chi=B}\bigg]\bigg|_{\phi=0}
	\end{align}
	By the fundamental theorem of calculus we find that the boundary term $R_0^3r_1^2\big|_{\chi=B}$ is in fact controlled by the energy $\mathcal{E}(\phi)$, yielding \eqref{linenest}. More precisely, it holds $r_0^{-1}r_1^2\leq\mathcal{E}(\phi)$ for all $\chi\in[0,B]$. Finally, the energy estimate \eqref{linenest2} for the rest of the linearised variables follows from \eqref{restlinvar} and  \eqref{linenest}.
\end{proof}
We may use \eqref{linenest}, along with the fact that the coefficients in \eqref{linHessreq6} are independent of $\phi$, to obtain higher order estimates  for the linearised variables to any order. 
\begin{prop}\label{prop:highenest}
The following higher order energies of the solution $(r_1$, $\rho_1$, $\psi_1$, $\omega_1$, $m_1)$ to the linearised equations \eqref{linHessreq}-\eqref{restlinvar} are bounded:
\begin{align}\label{highenest}
\mathcal{E}^{(i)}(\phi):=\int^B_0\sum_{j_1+j_2\leq i}r_0^{6j_2-2}(\partial_\phi^{j_1}\partial_\chi^{j_2} r_1) d\chi\leq C_i\mathcal{E}(0)
\end{align}
and
\begin{align}\label{highenest2}
\int_0^B\sum_{j_1+j_2\leq i}r_0^{6j_2}\big[(\partial_\phi^{j_1}\partial_\chi^{j_2}\psi_1)^2+(\partial_\phi^{j_1}\partial_\chi^{j_2}\omega_1)^2\big]d\chi\leq C_i\mathcal{E}(0)
\end{align}
for all $\phi\ge0,i\ge1$, where $C_i>0$ is increasing in $i$.
\end{prop}
\begin{proof}
We will only derive \eqref{highenest}. The second estimate follows easily by applying the first one to \eqref{restlinvar}. 

Moreover, the case $j_2\leq1$ follows immediately from the previous proposition by commuting the equation \eqref{linHessreq6} with $\partial_\phi^{i}$, since the coefficients of the wave equation are independent of $\phi$. On the other hand, solving for $\partial_\chi^2r_1$ in \eqref{linHessreq6} and using \eqref{domcoeffr1}-\eqref{lowcoeffr1} we obtain:
\begin{align}\label{d2r/dchi2}
\partial_\chi^2r_1=O(r_0^{-4})\partial^2_\phi r_1+O(r_0^{-3})\partial_\chi r_1+O(r^{-6}_0)r_1,
\end{align}
where the functions $O(r_0^k)$ are analytic and satisfy $\partial_\chi O(r_0^k)=O(\partial_\chi r_0^k)=O(r_0^{k-1}\partial_\chi r_0)=O(r_0^{k-3})$, $\partial_\phi O(r_0^k)=0$, see also Lemma \ref{lemma:stars:asymptotics}. Thus, we can derive the higher order energy estimate \eqref{highenest} for each summand inductively in $j_2\ge2$. Indeed, assume the corresponding estimate holds true for all summands up to a fixed number $j_2$ of $\partial_\chi$ derivatives (and all $j_1$). Commuting \eqref{d2r/dchi2} with $\partial_\chi^{j_2-1}$ we have
\begin{align}\label{commd2r/dchi2}
\partial_\chi^{j_2+1}r_1=\sum_{k=0}^{j_2-1}\bigg[O(r_0^{-4-3k})\partial^2_\phi\partial_{\chi}^{j_2-1-k} r_1+O(r_0^{-3-3k})\partial_\chi^{j_2-k} r_1+O(r^{-6-3k}_0)\partial_{\chi}^{j_2-1}r_1\bigg]
\end{align}
Hence, we obtain 
\begin{multline}\label{highdchienest}
\int^B_0 r^{6j_2+4}_0(\partial_\chi^{j_2+1}r_1)^2d\chi\leq 
\, C_{j_2}\sum_{k=0}^{j_2-1}\bigg[\|r^{3(j_2-1-k)+1}_0\partial_\phi^2\partial_{\chi}^{j_2-1-k}r_1\|_{L^2_\chi([0,B])}\\
+\|r^{3(j_2-k)-1}_0\partial_{\chi}^{j_2-k}r_1\|_{L^2([0,B]_\chi)}+\|r^{3(j_2-1-k)-1}_0\partial_{\chi}^{j_2-1-k}r_1\|_{L^2([0,B]_\chi)}\bigg]
\leq\,C_{j_2}\mathcal{E}(0)
\end{multline}
Note that only the second term  is at the level of \eqref{highenest}, whereas the first and third terms have smaller weights in $r_0$ than they could afford. Finally, commuting \eqref{commd2r/dchi2} with $\partial_\phi^{j_1}$ and repeating the above argument, we complete the finite induction proof of \eqref{highenest}. 
\end{proof}
In the following section we exhibit the existence of time-periodic solutions to the linearised system.
This confirms that the above energy \emph{boundedness} is indeed the best estimate for the linearised sytem that one can hope for in this setting.

\subsection{Periodic solutions}

In order to construct periodic solutions to the linearised equations (\ref{linHessreq},\ref{restlinvar}), it suffices to examine the wave equation satisfied by $r_1$, which as we showed above can be reduced to \eqref{linHessreq6}. We rewrite the latter using the spatial derivative $\partial_{r_0}$ instead of $\partial_\chi$:
\begin{subequations}\label{linHessreq7}
  \begin{equation}
    -\partial_\phi^2r_1 = H\,r_1
  \end{equation}
\begin{align}\label{eq:H:r1}
\notag H\,r_1:=&-\frac{e^{-\int^\chi_0fd\chi'}}{2\rho_0-1}\partial_\chi r_0\frac{\partial}{\partial{r_0}}\bigg[e^{\int^\chi_0fd\chi'}\frac{1-\frac{2m_0}{r_0}}{\partial_\chi r_0}\frac{\partial r_1}{\partial r_0} \bigg]\\
&-\frac{1}{2\rho_0-1}\bigg[2\frac{\partial_\chi\psi_0}{\partial_\chi r_0}\Bigl(\frac{m_0}{r^2_0}+4\pi r_0(\rho_0-1)\Bigr)+2\frac{m_0}{r^3_0}\\
\notag&+4\pi r_0(2\rho_0-1)\Bigl(\frac{2}{r_0}-\frac{\partial_\chi\psi_0}{\partial_\chi r_0}\Bigr)
+\frac{1-\frac{2m_0}{r_0}}{\partial_\chi r_0}\partial_\chi\Bigl(\frac{2}{r_0}-\frac{\partial_\chi\psi_0}{\partial_\chi r_0}\Bigr)\bigg]r_1
\end{align}
\end{subequations}
We first observe that in the small star regime, $H$ is a perturbation of
\begin{equation}\label{H_0}
H_0:=-\frac{1}{4\pi r_0^2}\partial_{r_0}(4\pi r_0^2\partial_{r_0})+\frac{2}{r_0^2}=-\partial_{r_0}^2-\frac{2}{r_0}\partial_{r_0}+\frac{2}{r_0^2}
\end{equation}

\begin{lemma}\label{lemma:H}
For small stars, we can decompose
\begin{subequations}
\begin{equation}\label{Hdec}
Hr_1=H_0r_1+H_1r_1,
\end{equation}
where $H_0$ is given by \eqref{H_0}, and 
\begin{equation}\label{H_1}
H_1=O(R_0^2)\partial_{r_0}^2+\frac{O(R_0^2)}{r_0}\partial_{r_0}+O(1)\,.
\end{equation}
\end{subequations}
\end{lemma}
\begin{proof}
  This follows by substituting the values from  Lemma~\ref{lemma:stars:asymptotics} in \eqref{eq:H:r1}, see also \eqref{domcoeffr1}-\eqref{lowcoeffr1}.
\end{proof}
\begin{rem}
  Note that for small stars, the leading order potential term is included in $H_0$, and the sign of the corresponding term in $H_1$ is irrelevant here. The former arises from \eqref{eq:H:r1} as $-(\partial_\chi r_0)^{-1}\partial_\chi(2/r_0)=2/r_0^2$.
\end{rem}

We begin by making the ansatz $r_1=e^{i\sqrt{\lambda} \phi}h(r_0)$. Then \eqref{linHessreq7} yields the eigenvalue problem:
\begin{align}\label{linperiodeq}
Hh=\lambda h
\end{align}
In the following we restrict ourselves to \emph{regular} solutions satisfying
\begin{equation}\label{linreg}
h(0)=0,\qquad \partial_{r_0}h(0)=1\,.
\end{equation}
\begin{rem}
This is justified by the requirement that $r(r_0)=r_0+r_1(r_0)$ gives rise to a spherically symmetric metric \eqref{eq:metric:com} which is $\mathrm{C}^1$ regular at the center, which implies $r(0)=0$, and $\partial_{r_0} r(0)=1$. 
\end{rem}

Next we observe that $H$ is a positive operator, which tells us that all eigenvalues are positive and real.
We omit here a formal definition of $H$ between appropriate Sobolev spaces, but the boundary conditions \eqref{bdcondr1} are of course essential here.
\begin{lemma}\label{lemma:lambda:pos}
  Let $h$ be a solution to \eqref{linperiodeq} satisfying the boundary conditions \eqref{linreg} and \eqref{bdcondh}. Then $\lambda\in(0,+\infty)$.
\end{lemma}

\begin{proof}
  We show that $H$ given by \eqref{eq:H:r1} is a positive operator.
First, integrating $h H_0h$ by parts on $[0,R_0]$, where $H_0$ is given by \eqref{H_0}, we find
\begin{align}
&\int^{R_0}_0h H_0h \,4\pi r_0^2\ud r_0= 4\pi\int_0^{R_0}\big[-h\partial_{r_0}(r_0^2\partial_{r_0}h)+2h^2\big]dr_0\\
\notag=&\,
4\pi\int_0^{R_0}\big[r_0^2(\partial_{r_0}h)^2+2h^2\big]dr_0-4\pi (r_0^2h\partial_{r_0}h)\big|^{R_0}_0\\
\tag{by \eqref{linreg},\eqref{bdcondh}}=&4\pi\int_0^{R_0}\big[r_0^2(\partial_{r_0}h)^2+2h^2\big]dr_0+4\pi R_0^2\big[8\pi R_0+O(R_0^3)\big]h^2(R_0)>0
\end{align}
  and this positivity property remains true  when we add the perturbation $H_1$. Thus, for small $R_0$, $\int \lambda h^2 >0$ and the conclusion follows.
\end{proof}


The possible values of $\lambda$ are then fixed by the boundary condition \eqref{bdcondr1} which reduces to
\begin{equation}\label{bdcondh}
\partial_{r_0}h=\frac{1}{\partial_{\chi}r_0}\Bigl(\frac{\partial_\chi\psi_0}{\partial_\chi r_0}-\frac{2}{r_0}\Bigr)h,\qquad\text{at $r_0=R_0$.}
\end{equation}
First, we study the leading order part of \eqref{linperiodeq}.
\begin{prop}\label{prop:periodsol}
The regular solutions to $H_0h=\lambda h$ are of the form:
\begin{equation}\label{hj}
h_j=\frac{3}{\sqrt{\lambda_j}}j_1(\sqrt{\lambda_j} r_0)\,,\qquad h(0)=0\,,\quad \partial_{r_0}h(0)=1,
\end{equation}
where $j_1(\sqrt{\lambda} r_0)$ denotes the spherical Bessel function \eqref{j1}, and $\{\lambda_j\}^\infty_1$, $\lambda_j\rightarrow+\infty$, is an increasing sequence of positive real numbers determined by the boundary condition \eqref{bdcondh}. Moreover, the smallest frequency is of the order $\sqrt{\lambda_1}\sim R_0^{-1}$, as $R_0\rightarrow0$. 
\end{prop}
\begin{proof}
Observe that the function
\begin{equation}\label{j1}
j_1(\sqrt{\lambda} r_0)=\frac{\sin(\sqrt{\lambda} r_0)}{(\sqrt{\lambda} r_0)^2}-\frac{\cos(\sqrt{\lambda} r_0)}{\sqrt{\lambda} r_0}
\end{equation}
satisfies the o.d.e.~$H_0j_1=\lambda j_1$ and $j_1(0)=0$, $\partial_{r_0}j_1(\sqrt{\lambda} r_0)\big|_{r_0=0}=\frac{\sqrt{\lambda}}{3}$. Hence, the solution $h$ that we seek equals:
\begin{equation}\label{leadhsol}
h(r_0)=\frac{3}{\sqrt{\lambda}}\bigg(\frac{\sin(\sqrt{\lambda} r_0)}{(\sqrt{\lambda} r_0)^2}-\frac{\cos(\sqrt{\lambda} r_0)}{\sqrt{\lambda} r_0}\bigg)
\end{equation}
According to the leading order expressions \eqref{leadasymb},\eqref{lowcoeffr1}, the boundary condition \eqref{bdcondh} yields the following equation for $\lambda$:
\begin{equation}\label{lambdaeq}
\bigl[2-8\pi R_0^2+O(R^4_0)\bigr]\bigg[\frac{\cos(\sqrt{\lambda} R_0)}{\sqrt{\lambda} R_0^2}
-\frac{\sin(\sqrt{\lambda} R_0)}{\lambda R_0^3}\bigg]
+\frac{\sin(\sqrt{\lambda} R_0)}{ R_0}=0
\end{equation}
It is evident that for large values of $\lambda\in(0,+\infty)$, holding $R_0$ fixed, the first two terms in \eqref{lambdaeq} become negligible compared to the last one. 
Thus, there exists a discrete set of zeros $\{\lambda_j\}^\infty_0$ of \eqref{lambdaeq} tending to infinity. Also, for small values of $\lambda\ge0$, $R_0$ being fixed and small, the equation \eqref{lambdaeq} to leading order reads:
\begin{equation}\label{lambdaeq2}
\bigl[2-8\pi R_0^2+O(R^4_0)\bigr]\Bigl(-\frac{1}{3}\sqrt{\lambda}+O(\lambda^\frac{3}{2}R_0^2)\Bigr)+\sqrt{\lambda}+O(\lambda R_0^3)=0
\end{equation}
or simply $\sqrt{\lambda}/3+O(\sqrt{\lambda} R_0^2)=0$, which cannot be satisfied for $R_0$ sufficiently small.

The last part of the proposition is a question of computing the first term in the expansion of $\lambda_1:=\lambda_1(R_0)$.
We search for the leading order exponent $a$ of $\sqrt{\lambda_1} R_0\sim R_0^{a}$, as $R_0\rightarrow0$. For $a=0$, all terms in \eqref{lambdaeq} are of the same order and hence a solution $\lambda$ with that leading order exists. In order to show that it corresponds to  $\lambda_1$, it suffices to argue that a lower order with $a>0$ is not possible. Indeed, this is the case, since for $a>0$ \eqref{lambdaeq2} is still valid and it yields:
\begin{align}\label{lambdaeq3}
\frac{\sqrt{\lambda_1}}{3}+O(R_0^{-1+3a})+O(R_0^{1+a})+O(R_0^{1+2a})=0,
\end{align}
which also does not admit a solution of the form  $\sqrt{\lambda_1}\sim R_0^{-1+a}$, $a>0$, as $R_0\rightarrow0$.
\end{proof}  
The existence of periodic solutions to the original wave equation \eqref{linHessreq7} for $r_1$ now follows by perturbation theory; see e.g.~\S13~in~\cite{Courant:31}, and recall also Lemma~\ref{lemma:lambda:pos}. 
\begin{cor}\label{cor:periodsol}
There exist periodic solutions to \eqref{linHessreq7}, $r_{1j}=e^{i\sqrt{\tilde{\lambda}_j}\phi}\tilde{h}_j(r_0)$, for
a discrete set $\{\tilde{\lambda}_j\}^\infty_0\subset(0,+\infty)$, and $\tilde{h}_j$ satisfying $\tilde{h}_j(0)=0$, $\partial_{r_0}\tilde{h}_j(0)=1$, and \eqref{bdcondh}. Moreover, $\tilde{\lambda}_j=\lambda_j+\mathcal{O}(R_0^2)$, where the $\lambda_j$'s are given by Proposition \ref{prop:periodsol}.
\end{cor}
\begin{rem}
  In view of \eqref{lambdaeq}, the eigenvalues $\tilde{\lambda}_j$ are \emph{simple} for large $j\in\mathbb{N}$, 
and the eigenfunctions $\tilde{h}_j$ of the operator $H$ are $O(R_0^2)$ perturbations of the $h_j$'s given by \eqref{hj}.
\end{rem}
%
%

%
%





\providecommand{\bysame}{\leavevmode\hbox to3em{\hrulefill}\thinspace}
\providecommand{\MR}{\relax\ifhmode\unskip\space\fi MR }
\providecommand{\MRhref}[2]{%
  \href{http://www.ams.org/mathscinet-getitem?mr=#1}{#2}
}
\providecommand{\href}[2]{#2}

\bigskip
Grigorios Fournodavlos\\
{\itshape Address:} {\scshape\footnotesize LJLL, Sorbonne Universit\'e, 4 place Jussieu, 75005 Paris, France }\\
{\itshape Email:} {\ttfamily grigorios.fournodavlos@sorbonne-universite.fr}

\smallskip
Volker Schlue\\
{\itshape Address:} {\scshape\footnotesize University of Melbourne, Parkville, VIC 3010, Australia}\\
{\itshape Email:} {\ttfamily volker.schlue@unimelb.edu.au}

\end{document}